\documentclass{article}
\usepackage[utf8]{inputenc}
\usepackage[left=3cm,right=3cm,top=2cm,bottom=2cm]{geometry}

\usepackage{comment}
\usepackage{caption}
\usepackage{subcaption}
\usepackage{graphicx}
\usepackage{amsfonts,amsmath,amssymb,amsthm}
\usepackage{dsfont}
\usepackage[dvipsnames]{xcolor}
\usepackage{pgfplots}
\usepackage{hyperref}
\usepackage[normalem]{ulem}

\newtheorem{theorem}{Theorem}

\newtheorem{lemma}[theorem]{Lemma}
\newtheorem{definition}[theorem]{Definition}
\newtheorem{proposition}[theorem]{Proposition}
\newtheorem{remark}[theorem]{Remark}
\newtheorem{example}[theorem]{Example}

\newtheorem{question}{Question}

\providecommand{\keywords}[1]
{
  \small	
  \textbf{\textbf{Keywords}:} #1
}

\newcommand{\subjclass}[1]
{
  \small
  \textbf{\emph{Mathematics subject classification:}} #1
}

\newcommand{\ee}{\varepsilon}
\newcommand{\EE}{\mathbb{E}}

\newcommand{\PP}{\mathbb{P}}
\newcommand{\RR}{\mathbb{R}}
\newcommand{\CC}{\mathbb{C}}

\newcommand{\R}{\mathbb{R}}

\newcommand{\dint}{\mathrm{d}}

\usepackage[flushleft]{threeparttable}
\usepackage{array,booktabs,makecell}
\usepackage[font=small]{caption}

\title{Operatopes, Operanoids, and Noncommutative Zonoids}
\author{Eliza O'Reilly$^\dag$ and Venkat Chandrasekaran$^\ddag$ \thanks{Email: \texttt{eoreilly@jhu.edu,venkatc@caltech.edu}} \vspace{0.25in} \\ $^\dag$Department of Applied Mathematics and Statistics \\ Johns Hopkins University \\ Baltimore, MD 21218 \\ \\ $^\ddag$ Department of Computing and Mathematical Sciences \\ Department of Electrical Engineering \\ California Institute of Technology \\ Pasadena, CA 91125}

\date{\today}

\begin{document}

\maketitle

\begin{abstract}
We study a class of convex bodies called \emph{operatopes} that are obtained by taking Minkowski sums of affine images of an operator norm ball. This notion generalizes that of zonotopes which are Minkowksi sums of line segments. Taking the limit of the number of line segments to infinity yields the class of convex bodies called zonoids, which can also be viewed as the expectation of a random line segment.  Expanding on this interpretation, we analogously define \emph{operanoids} as the expectation of a random affine image of an operator norm ball.  In studying the properties of operanoids when the dimension of the operator norm ball grows, we arrive at a new asymptotic regime for limits of convex bodies.  This leads to the more general class of convex bodies called \emph{noncommutative zonoids}, and we use the framework of free probability theory to illustrate basic properties and examples. Finally, we discuss applications of operanoids and noncommutative zonoids in statistics and stochastic processes.
\end{abstract}

% REQUIRED

% REQUIRED
%\begin{AMS}
%Primary ; secondary ; 
%\end{AMS}

\section{Introduction}\label{sec:intro}

A \emph{zonotope} is a finite Minkowski sum of line segments.
%\VC{Do we want to say centered?} \EO{By centered do you mean origin symmetric? I would say we wait to specify the center at the origin til later. The main reason is that lift zonoids are actually not centered at the origin. But since all zonoids/operanoids have a centre of symmetry, we can make this assumption through the rest of the paper.} \VC{Let's discuss this.} \emph{Zonoids} are convex sets that are approximable in the Hausdorff metric by zonotopes. 
Zonotopes and their limits, called \emph{zonoids}, are prominent families of convex bodies that arise in numerous areas of mathematics.  In polyhedral combinatorics, zonotopes represent a class of polyhedra that can be concisely described even though they may have exponentially many vertices and facets \cite{mcmullen1971zonotopes, ziegler2012lectures}.  In statistics, zonoids provide a useful geometric summary of high-dimensional data \cite{mosler_multivariate_2002}.  In stochastic geometry, zonoids characterize fundamental properties of random hyperplane tessellations \cite{hug2024poisson}.  More broadly, zonotopes have also proven useful in the development of practical methods and algorithms in an array of engineering problems  \cite{alamo_guaranteed_2005, althoff_zonotope_2011, bern_optimization_2001, guibas_zonotopes_nodate}.  Despite this range of applications, a basic limitation of zonotopes is their polyhedrality, which constrains their utility in settings where convex bodies with richer boundary structure are required to model phenomena of interest.  Zonoids are not polyhedral in general, but we do not have a good finitely-parameterized way to describe arbitrary zonoids and it is difficult to decide whether a general convex body is a zonoid (this is the so-called zonoid problem \cite{Bolker_Zonoid_Problem}).

\subsection{Our Contributions}
% \VC{Should we have a new subsection here titled ``Our Contributions''?}
% \EO{Sure! that makes sense.}

In this paper, we introduce a finitely-parametrized yet non-polyhedral generalization of zonotopes called operatopes.  Formally, an \emph{$m$-operatope} is a finite Minkowski sum of affine images of the operator norm ball in the space of $m \times m$ Hermitian matrices $\mathbb{H}^m$.  For $m=1$ the associated operator norm ball is equal to the line segment $[-1,1]$, so $1$-operatopes correspond to zonotopes.  However, $m$-operatopes for larger $m$ can be non-polyhedral in general.  As an example, the Euclidean ball of dimension $d$ is clearly not a zonotope, but it is a $(d+1)$-operatope.  Figure~\ref{fig:enter-label} gives other examples. Analogously to the definition of zonoids, we say that a convex body is an \emph{$m$-operanoid} if it can be approximated in the Hausdorff metric by $m$-operatopes.  In principle, one can also consider Minkowski sums of elementary objects other than images of operator norm balls (see Section~\ref{sec:related}) to obtain different non-polyhedral generalizations of zonotopes.  However, as we will see in the sequel, operatopes and operanoids represent appealing objects of investigation for several reasons.  Specifically, operanoids provide a natural `random-matrix' generalization of zonoids as they facilitate convex-geometric summaries of tuples of random matrices, much like zonoids summarize random vectors.  Moreover, our definition of operatopes and operanoids introduces the parameter $m$ that opens up the new asymptotic regime $m \to \infty$ for studying limits of convex bodies, and this regime is not present with zonotopes and zonoids.  Analyzing these limits requires a new perspective and, as with the study of limits of growing-sized random matrices, the theory of noncommutative probability furnishes the relevant tools.
% \VC{Is this last bit reasonable?  I think we need something like this to motivate why we consider our particular Minkowski sums.  For example, one might wonder at this early stage why operatopes are more interesting / special than discotopes or something similar?  What do you feel?} \EO{Yes definitely, I think this motivation section sounds great.}

% logical outline -- equivalent characterizations; examples inspired by random matrix ensembles.  boundary structure of operanoids; relation to general convex bodies; approximation

In Section~\ref{sec:operatopes} we formally define operatopes and operanoids, and we give several equivalent characterizations of operanoids from probabilistic and functional-analytic perspectives.  These generalize the corresponding characterizations for zonoids.  In particular, Vitale \cite{Vitale1991} showed that a convex body $Z$ that is centrally symmetric with respect to the origin is a zonoid in $\R^d$ if and only if its support function $h_Z$ can be expressed as follows:
\begin{equation*}
    h_Z(u) = \mathbb{E}[|\langle u, X \rangle |], ~~~ u \in \mathbb{S}^{d-1}
\end{equation*}
for a random vector $X$ in $\R^d$. Stated differently, a zonoid is the expected value of a random line segment that is centered at the origin.\footnote{The expected value of a random convex body is defined as the deterministic convex body with support function given by the expected value of the random support function \cite{Artstein_Vitale_1975}.}  We generalize this result in Theorem~\ref{t:zonoid_vitale_characterization} in which we show that an \emph{operanoid} that is origin-symmetric can equivalently be defined as the expectation of a random linear image of an operator norm ball.  This characterization yields yet another way to define operanoids in Theorem~\ref{t:L1-subspace}, which generalizes a description on zonoids based on subspaces of $L_1$ \cite{bolker_class_1969}.  We conclude this section with examples of operanoids and operatopes derived from prominent random matrix ensembles such as Wigner and Wishart random matrices.

Section~\ref{sec:properties} presents various properties of operanoids by building on the different characterizations described in Section~\ref{sec:operatopes}.  Specifically, we give examples of $2$-operatopes in $\R^3$ that are not zonoids, and furthermore, we discuss the relationship between operanoids and general symmetric convex bodies.  We also show that the faces of operanoids must themselves be operanoids in Proposition~\ref{prop:operanoid_faces}.  Finally, we present precise rates that quantify how well operanoids can be approximated by operatopes in Theorem~\ref{t:approx_rate}, which generalizes a result on the approximation of zonoids by zonotopes \cite{bourgain_approximation_1989}.

As we highlighted earlier, our definition of an operanoid leads to the regime $m \rightarrow \infty$ for studying limits of convex bodies.  To facilitate this investigation, in Section~\ref{sec:NCzonoids-overview} we consider tuples of noncommuting random variables in a $C^\star$-probability space and we define convex bodies called \emph{noncommutative zonoids} whose support function is specified in terms of the tracial state applied to a linear combination of the noncommuting variables; see Section~\ref{sec:NCrandomvariables} for preliminaries on noncommutative probability.  When the underlying noncommutative random variables are freely independent, the resulting noncommutative zonoid is called a \emph{free zonoid}.  We prove in Proposition~\ref{prop:freelimit} that any free zonoid can be approximated in the Hausdorff metric by a sequence of $m$-operanoids as $m \rightarrow \infty$.

In Section~\ref{sec:applications}, we take initial steps towards applications by illustrating the role of operanoids in two contexts.  First, we use operanoids and NC zonoids to construct convex-geometric summaries of high-dimensional data distributions by generalizing the notion of a lift zonoid.  Second, we show that operanoids can describe particular geometric properties of a random hypersurface model, in analogy to zonoidal characterizations of random hyperplane tessellations.

\begin{figure}
    \centering
    \includegraphics[width=0.25\linewidth]{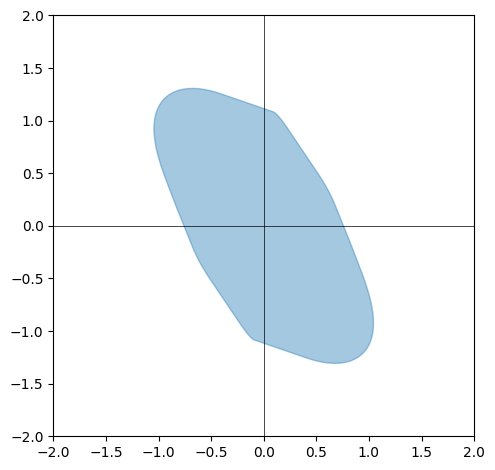} \hfill\includegraphics[width=0.25\linewidth]{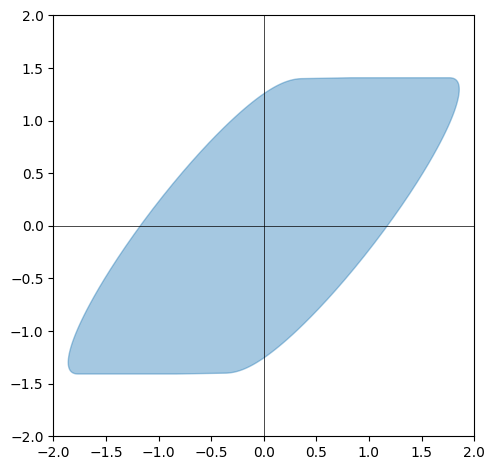} \hfill
    \includegraphics[width=0.25\linewidth]{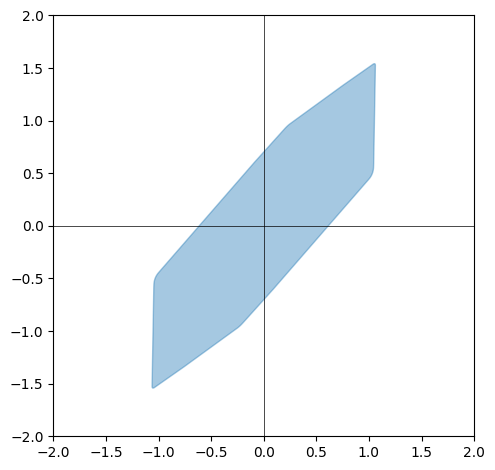}
    \caption{$3$-operatopes in $\R^2$ with one summand and the parametrizing matrices drawn from the GOE ensemble.}
    \label{fig:enter-label}
\end{figure}

\subsection{Related Work}\label{sec:related}

To discuss the relation between our work and the previous literature, it is useful to recall some ideas from convex geometry \cite{schneider_2013}.
% \VC{Should we give a background reference here?} \EO{Yes, I think the Schneider book is a good one to have here}
A \emph{Minkowski class} of convex bodies in $\RR^d$ is a non-empty collection of convex bodies that is closed under Minkowski addition and dilation. For a subset $G$ of affine transformations on $\RR^d$, a Minkowksi class $\mathcal{M}$ is $G$-invariant if $K \in \mathcal{M}$ implies $gK \in \mathcal{M}$ for any $g \in G$. This notion leads to the following definition: for a set of convex bodies $\mathcal{K}$ in $\RR^d$ and a subset $G$ of the affine transformation group of $\RR^d$, the \emph{$G$-invariant Minkowksi class generated by $\mathcal{K}$} is the collection of convex bodies of the form
\begin{align*}
    t_1 g_1K_1 + \cdots t_ng_nK_n, \text{ where } n \in \mathbb{N} \text{ and } t_i \geq 0, \, g_i \in G, \,  K_i \in \mathcal{K} \text{ for all } i = 1, \ldots, n,
\end{align*}
and their limits with respect to the Hausdorff metric.

Several families of convex bodies from previous work can be viewed as particular instances of Minkowski classes.  Zonoids are the smallest nontrivial example of an affine invariant Minkowksi class, namely one that is generated by line segments.  Other examples of affine invariant Minkowski classes include polyoids \cite{hug2023extremizers}, which have been considered when finding extremals of the Alexandrov-Fenchel inequality, and discotopes \cite{gesmundo_geometry_2022} which take Minkowski sums of ellipsoids.  Discotopes are zonoids but polyoids in general are not zonoids.

Operanoids also constitute a particular affine invariant Minskowski class of convex bodies.  Specifically, letting $\mathcal{K} \subset \R^d$ be the collection of all affine images of the operator norm ball in $\mathbb{H}^m$, the set of $m$-operanoids in $\R^d$ is precisely the associated Minskowski class with respect to this $\mathcal{K}$.

% Examples of other affine invariant Minkowski classes include polyoids \cite{hug2023extremizers}, which have been considered when finding extremals of the Alexandrov-Fenchel inequality, and discotopes \cite{gesmundo_geometry_2022} which take Minkowski sums of ellipsoids. 

% \EO{A \emph{Minkowski class} of convex bodies in $\RR^d$ is a non-empty collection of convex bodies that is closed under Minkowski addition and dilation. For a subset $G$ of affine transformations on $\RR^d$, a Minkowksi class $\mathcal{M}$ is $G$-invariant if $K \in \mathcal{M}$ implies $gK \in \mathcal{M}$ for any $g \in G$. This notion leads to the following definition: for a set of convex bodies $\mathcal{K}$ in $\RR^d$ and a subset $G$ of the affine transformation group of $\RR^d$, the \emph{$G$-invariant Minkowksi class generated by $\mathcal{K}$} is the collection of convex bodies of the form
% \begin{align*}
%     t_1 g_1K_1 + \cdots t_ng_nK_n, \text{ where } n \in \mathbb{N}, t_i \geq 0, g_i \in G, K_i \in \mathcal{K} \text{ for all } i = 1, \ldots, n
% \end{align*}
% and their limits with respect to the Hausdorff metric. 
% Zonoids are the smallest nontrivial example of an affine invariant Minkowksi class that is generated by line segments.
% Examples of other affine invariant Minkowski classes include polyoids \cite{hug2023extremizers}, which have been considered when finding extremals of the Alexandrov-Fenchel inequality, and discotopes \cite{gesmundo_geometry_2022} which take Minkowski sums of ellipsoids. }

\subsection{Notation and Terminology}\label{sec:notation}
% \VC{My sense is that this should be here rather than in the next section.  What do you think?}

We let $\mathbb{V}$ denote a Euclidean space with inner product $\langle \cdot , \cdot \rangle$, and $\mathbb{S}(\mathbb{V})$ and $B(\mathbb{V})$ denote the unit sphere and unit ball of $\mathbb{V}$, respectively. As shorthand, $\mathbb{S}^{d-1} := \mathbb{S}(\RR^d)$ and $B^d := B(\RR^d)$ denote the unit sphere and unit ball in $\RR^d$.  A nonempty compact convex set $K \subset \mathbb{V}$ is called a convex body.\footnote{In some texts, convex bodies are also required to have nonempty interior and for the origin to belong to the interior.  Here we do not make these assumptions.}  The set of all convex bodies in $\mathbb{V}$ is denoted $\mathcal{K}(\mathbb{V})$, with the shorthand $\mathcal{K}^d := \mathcal{K}(\RR^d)$ denoting all convex bodies in $\RR^d$. In the following we will mostly be focused on centrally symmetric convex bodies, i.e. $K \in \mathcal{K}(\mathbb{V})$ such that there exists some center $c \in \mathbb{V}$ for which $K-c = -(K-c)$. Centrally symmetric convex bodies form a closed subset of $\mathcal{K}(\mathbb{V})$.  The \emph{support function} of a convex body $K \in \mathcal{K}(\mathbb{V})$ is denoted $h_K : \mathbb{V} \to \RR$ and is defined as
\[h_K(u) := \max_{x \in K} \, \langle u, x \rangle.\]
Support functions of convex bodies are sublinear, i.e. $h_{K}(\alpha u) = \alpha h_K(u)$ for $\alpha \geq 0$ and $h(u + \omega) \leq h_K(u) + h_{K}(\omega)$, and every sublinear function is also the support function of a convex body.  In particular, support functions uniquely characterize convex bodies.

%For a Euclidean space $\mathbb{V}$ and linear image $A: \EE \to \mathbb{V}$, $A(K)$ is a convex body in $\mathbb{V}$ with support function $h_{A(K)}(v) = h_K(A^tv)$.

Many properties of and operations on the set of convex bodies can be described in terms of their support functions. The \emph{Hausdorff metric} between convex bodies $K, L \in \mathcal{K}(\mathbb{V})$ is defined by 
\[d_H(K, L) := \max_{u \in S(\mathbb{V})} |h_K(u) - h_L(u)|,\]
and $\mathcal{K}(\mathbb{V})$ equipped with the metric $d_H$ is a complete metric space \cite{schneider_2013}.
The \emph{Minkowski sum} of two convex bodies $K, L \in \mathcal{K}(\mathbb{V})$ is given by 
$K + L = \{x + y \in \mathbb{V} \, | \, x \in K, y \in L\}$,
and the support function of $K + L$ is the sum of support functions, i.e., $h_{K + L}(u) = h_K(u) + h_L(u).$  We refer to the standard reference \cite{schneider_2013} for further background on the geometry of convex bodies.

In the following, we will appeal to the notion of a \emph{random convex body} in $\mathbb{V}$, which is a measurable map $Z$ from a probability space $\Omega$ to $\mathcal{K}(\mathbb{V})$. Under the assumption $\mathbb{E}[\max_{z \in Z} \|z\|_2] < \infty$ for the Euclidean norm $\|\cdot\|_2$ on $\mathbb{V}$, the expectation of the support function $\EE[h_Z(u)]$ is a well-defined sublinear function, and thus is the support function of a convex body. We then define the \emph{expectation of Z} to be the convex body $\EE[Z]$ with support function
\begin{align*}
    h_{\mathbb{E}[Z]}(u) := \mathbb{E}[h_Z(u)].
\end{align*}

Finally, the Euclidean space of Hermitian $m \times m$ matrices is denoted $\mathbb{H}^m$.  The \emph{operator norm} $\|A\|_{op}$ of any $A \in \mathbb{H}^m$ is the largest singular value of $A$.  The unit ball of the operator norm on $\mathbb{H}^m$ is given by $$B_\infty(\mathbb{H}^m) := \{A \in \mathbb{H}^m : \|A\|_{op} \leq 1\}.$$  
The \emph{trace (or nuclear) norm} $\|A\|_1$ of any $A \in \mathbb{H}^m$ is given by the sum of the singular values of $A$.  The support function of the operator norm ball $B_{\infty}$ is $h_{B_{\infty}}(A) = \|A\|_1$ for $A \in \mathbb{H}^m$.  A linear map $\mathcal{A}: \mathbb{H}^m \to \RR^d$ is determined by a $d$-tuple of matrices $A_1, \ldots, A_d \in \mathbb{H}^m$ so that for $X \in \mathbb{H}^m$, we have $\mathcal{A}(X) = (\langle A_1, X\rangle, \ldots, \langle A_d, X\rangle)$ and the adjoint is given by $\mathcal{A}^\star u = \sum_{i=1}^d u_iA_i$ for $u \in \RR^d$.  The support function of $\mathcal{A}(B_{\infty})$ is then
 \begin{align*}
  h_{\mathcal{A}(B_{\infty})}(u) := h_{B_{\infty}}(\mathcal{A}^\star u) = \left\|\sum_{i=1}^d u_i A_i\right\|_1.   
 \end{align*}

\section{Operatopes and Operanoids} \label{sec:operatopes}
In this section, we formally present and investigate the properties of operatopes and operanoids.  Section~\ref{sec:definition} gives the definition of these objects.  We also  derive several equivalent characterizations of operanoids in Section~\ref{sec:definition}, paralleling those that are available for zonoids. Section~\ref{sec:rm-examples} presents examples of operanoids derived from various random matrix ensembles, and we revisit these in Section~\ref{sec:NCzonoids-overview}.  

\subsection{Definition}\label{sec:definition}

As we remarked in Section \ref{sec:intro}, operatopes are akin to zonotopes and operanoids are obtained by taking limits of operatopes.  Here we give the formal definitions.

\begin{definition}
A centrally symmetric convex body $Z$ in $\RR^d$ that is the finite Minkowski sum of affine images of the operator norm ball in $\mathbb{H}^m$ is called an $m$-operatope.  Equivalently, an $m$-operatope $Z$ in $\R^d$ is a convex body with support function of the following form:
\[h_{Z}(u) = \left\langle v, u \right\rangle + \sum_{j=1}^N  \left\|\sum_{i=1}^d u_i A^{(j)}_i\right\|_1, \quad u \in \mathbb{S}^{d-1},\]
for some $N \geq 1$ and $A^{(j)}_1, \ldots, A^{(j)}_d \in \mathbb{H}^m$ for $j = 1, \ldots, N$ and $v \in \R^d$. If $Z$ is centered at the origin, then $v = 0$.

A centrally symmetric convex body $Z$ in $\RR^d$ is an \emph{$m$-operanoid} if it is the limit, in the Hausdorff sense, of $m$-operatopes.

\end{definition}

We note that $1$-operatopes are zonotopes and similarly $1$-operanoids are zonoids.  We describe next the structure of $2$-operatopes in $\R^d$ that are defined by $2 \times 2$ real symmetric matrices, i.e., the matrices $A_1^{(j)},\dots,A_d^{(j)}$ in the above definition are $2 \times 2$ real symmetric matrices.  In this setting, we can directly compute the relevant eigenvalues, and thus obtain the general form of the support function.  Specifically, such $2$-operatopes are Minkowski sums of convex bodies with a support function of the following form for some vectors $a,b,c \in \RR^d$:
\begin{align*}
    h(u) = \left\|\sum_{i=1}^d u_i A_i \right\|_1 = \left\|\begin{bmatrix} \sum_{i=1}^d u_i a_i & \sum_{i=1}^d u_i b_i \\ \sum_{i=1}^d u_i b_i & \sum_{i=1}^d u_i c_i \end{bmatrix} \right\|_1 = \left\|\begin{bmatrix} \langle u,a \rangle & \langle u, b \rangle \\ \langle u, b \rangle & \langle u, c \rangle \end{bmatrix} \right\|_1.
\end{align*}
The eigenvalues of the above matrix are
\begin{align*}
    \lambda_1, \lambda_2 %&= \frac{\langle u, a+ c \rangle}{2} \pm \sqrt{\frac{(\langle u, a+ c \rangle)^2}{4} - \langle u, a\rangle\langle u, c \rangle + \langle u,b \rangle^2} \\
    %&= \frac{\langle u, a+ c \rangle}{2} \pm \sqrt{\frac{\langle u, a \rangle^2+ \langle u,c \rangle^2 + 2\langle u,a\rangle \langle u, c \rangle}{4} - \langle u, a\rangle\langle u, c \rangle + \langle u,b \rangle^2} \\
    %&= \frac{\langle u, a+ c \rangle}{2} \pm \sqrt{\frac{\langle u, a \rangle^2+ \langle u,c \rangle^2 - 2\langle u,a\rangle \langle u, c \rangle}{4} + \langle u,b \rangle^2} \\
    &= \frac{\langle u, a + c \rangle}{2} \pm \sqrt{\frac{\langle u, a - c \rangle^2}{4} + \langle u,b \rangle^2},
\end{align*}
and thus this linear image of the operator norm ball in the space of $2 \times 2$ matrices must have support function of the form
\begin{align}\label{eq:2-operatope}
    h(u) &= |\lambda_1| + |\lambda_2| = \left|\frac{\langle u, a + c \rangle}{2} + \sqrt{\frac{\langle u, a - c \rangle^2}{4} + \langle u,b \rangle^2}\right| + \left|\frac{\langle u, a + c \rangle}{2} - \sqrt{\frac{\langle u, a - c \rangle^2}{4} + \langle u,b \rangle^2}\right| \notag \\
    %&= \left|\frac{\langle u, a + c \rangle}{2}\right| + \sqrt{\frac{\langle u, a - c \rangle^2}{4} + \langle u,b \rangle^2} \notag \\
    %& \qquad + \max\left\{\left|\frac{\langle u, a + c \rangle}{2}\right| - \sqrt{\frac{\langle u, a - c \rangle^2}{4} + \langle u,b \rangle^2}, -\left|\frac{\langle u, a + c \rangle}{2}\right| + \sqrt{\frac{\langle u, a - c \rangle^2}{4} + \langle u,b \rangle^2} \right\} \notag \\
    &= \max\left\{\left|\langle u, a + c \rangle\right| , \sqrt{\langle u, a - c \rangle^2 + \langle u, 2b \rangle^2} \right\}. 
\end{align}
In summary, $2$-operatopes specified by real symmetric matrices are Minkowski sums of convex hulls of a line segment and a $2$-dimensional disc.  As we will see in Section \ref{sec:relation}, there are $2$-operatopes in $\R^3$ that are not zonoids (see Figure \ref{fig:nonzonoid}).

Even for operatopes that may be zonoids, the operatope representation can offer some advantages.  For example, an ellipsoid is a zonoid, but it does not have a finite parametrization within this class as a zonotope because it is not a polytope.  On the other hand, ellipsoids in $\R^d$ are $(d+1)$-operatopes, which facilitates a finite parametrization via linear images of operator norm balls.  More concretely, let $C \in \RR^{d \times d}$ and consider the convex body in $\RR^d$ with support function of the form $h_K(u) = \|Cu\|_2$, which is an ellipsoid in $\RR^d$ centered at the origin.  Define $A_1, \ldots, A_d \in \mathbb{H}^{d+1}$ such that 
\begin{align}
    \sum_{i=1}^d u_i A_i = \frac{1}{\sqrt{2}}\begin{bmatrix} 0 & & & (Cu)_1 \\ & \ddots & &  \vdots \\ & & 0 & (Cu)_d \\ (Cu)_1 & \cdots & (Cu)_d & 0 \end{bmatrix}.
\end{align}
Letting $\mathcal{A}: \mathbb{H}^{d+1} \to \RR^d$ denote the linear image such that $\mathcal{A}(X) = (\langle A_1, X \rangle, \cdots, \langle A_d, X\rangle)$,
\begin{align}
    h_{\mathcal{A}(B_{\infty})}(u) &= \left\|\sum_{i=1}^d u_i A_i \right\|_{1} = \sqrt{\frac{1}{2}\mathrm{Tr}\begin{bmatrix} (Cu)_1^2 & & (Cu)_1(Cu)_d & 0 \\ & \ddots & &  \vdots \\ (Cu)_d(Cu)_1 & & (Cu)_d^2 & 0 \\ 0 & \cdots & 0 & \sum_{i=1}^d (Cu)_i^2 \end{bmatrix}} = \|Cu\|_2.
\end{align}
Thus, ellipsoids are examples of linear images of operator norm balls.  More generally, finite Minkowski sums of ellipsoids, which are called \emph{discotopes} \cite{gesmundo_geometry_2022}, are also operatopes.

As we mentioned in the introduction, zonoids can be described as expectations of random line segments due to a result of Vitale \cite{Vitale1991}.  Here we generalize this result to show that operanoids can similarly be described as the expectation of a random linear image of an operator norm ball, with our proof closely following that of Theorem 3.1 in \cite{Vitale1991} for zonoids.

\begin{theorem}\label{t:zonoid_vitale_characterization}
A centrally symmetric convex body $Z$ in $\RR^d$ centered at the origin is an $m$-operanoid if and only if the support function of $Z$ is of the form: 
\[h_{Z}(u) = \EE\left[\left\|\sum_{i=1}^d u_i A_i \right\|_1\right],\]% = \int_{\Omega} \|\sum_{i=1}^d u_i A_i(\omega)\|_{1} \dint \omega.\]
for some $d$-tuple of random matrices $A_1, \ldots, A_d$ in $\mathbb{H}^m$ with $\EE\left[\left(\sum_{i=1}^d \|A_i\|_2^2\right)^{\frac{1}{2}}\right] < \infty$.
\end{theorem}

\begin{proof}
%First, recall that any $m$-operatope is the Minkowski sum $\mathcal{A}^{(1)}(B^m_{\infty}) + \cdots + \mathcal{A}^{(n)}(B^m_{\infty})$ in $\RR^d$ for some $n \in \mathbb{N}$, where the linear images $\mathcal{A}^{(j)}: M_m(\RR) \to \RR^d$ are defined by
%\[\mathcal{A}^{(j)}(B) = (\langle A^{(j)}_1, B\rangle, \ldots, \langle A^{(j)}_d, B \rangle) \in \RR^d, \quad  \text{ for all }B \in M_m(\RR),\]
%for some matrices $A^{(j)}_1, \ldots A^{(j)}_d \in M_m(\RR)$.
%Thus, 

For the forward implication, first consider the case that the $m$-operanoid $Z$ is in fact an $m$-operatope. Then, there exist $d$-tuples of $m \times m$ Hermitian matrices $(A_1^{(j)}, \ldots, A_d^{(j)})$ for $j = 1, \ldots, N$ and some $N \in \mathbb{N}$ such that the support function of $Z$ satisfies
\[h_Z(u) = \sum_{j=1}^N \left\|\sum_{i=1}^d u_i A_i^{(j)}\right\|_1, \quad u \in \mathbb{S}^{d-1}.\]
Then, we can let $(A_1, \ldots, A_d)$ be a $d$-tuple of random matrices such that $(A_1, \ldots, A_d) = N \cdot (A^{(j)}_1, \ldots A^{(j)}_d)$ with probability $\frac{1}{N}$, and see that $Z$ is the expectation of the \emph{random} linear image $\mathcal{A}$ of $B_{\infty}$ defined by $(A_1, \ldots, A_d)$. 

Now assume $Z$ is a general centrally symmetric $m$-operanoid and thus there exists a sequence of centrally symmetric $m$-operatopes $\{Z_N\}_{N \in \mathbb{N}}$ such that $Z_N$ converges to $Z$ with respect to the Hausdorff metric. For each $N \in \mathbb{N}$ we can write 
\[Z_N = \mathcal{A}_N^{(1)}(B_{\infty}) + \cdots + \mathcal{A}_{N}^{(k_n)}(B_{\infty}),\]
for some $k_N \in \mathbb{N}$ and linear images $\mathcal{A}_N^{(j)}:\mathbb{H}^m \to \RR^d$.
Define the norm $\|\cdot\|$ for $\mathcal{A}: \mathbb{H}^m\to \RR^d$ given by $\|\mathcal{A}\|^2 : = \sum_{i=1}^d \|A_i\|_{2}^2$. Then, let $\tilde{\mathcal{A}}_n$ be a random linear image such that $\tilde{\mathcal{A}}_n = \frac{\sum_{j=1}^{k_N} \|\mathcal{A}_N^{(j)}\|}{\|\mathcal{A}_N^{(j)}\|}\mathcal{A}_N^{(j)}$ with probability $p^{(N)}_j = \frac{\|\mathcal{A}_N^{(j)}\|}{\sum_{j=1}^{k_N} \|\mathcal{A}_N^{(j)}\|}$ and observe that
$$Z_N := \EE[\tilde{\mathcal{A}}_N(B_{\infty})].$$ 
Since $Z_N$ converges in Hausdorff distance to an $m$-operanoid $Z$, the sequence of $d$-tuples of random matrices $(\tilde{A}_{N,1}, \ldots, \tilde{A}_{N,d})$ defining $\tilde{\mathcal{A}}_N$ must be uniformly bounded for all $N$. By Prokorov's theorem, there is a subsequence $(\tilde{A}_{N_{\ell},1}, \ldots, \tilde{A}_{N_{\ell},d})$ that converges in distribution to a $d$-tuple $(\tilde{A}_1, \ldots, \tilde{A}_d)$ as $\ell \to \infty$. Thus, for all $u \in \mathbb{S}^{d-1}$,
\begin{align*}
    h_Z(u) = \lim_{\ell \to \infty} h_{\tilde{Z}_{n_\ell}}(u) = \lim_{\ell \to \infty} \EE\left[\left\|\sum_{i=1}^d u_i \tilde{A}_{N_\ell,i}\right\|_1\right] =  \EE\left[\left\|\sum_{i=1}^d u_i \tilde{A}_i\right\|_1\right].
\end{align*}
%and by Theorem 1.8.15 in \cite{}, $\EE[\|\sum_{i=1}^d u_i \tilde{A}_i\|_1]$ is the support function of a convex body $\tilde{Z}$ and $\tilde{Z}_n \to \tilde{Z}$ in the Hausdorff topology. Finally, by assumption, $\tilde{Z}_n$ converges to $\frac{1}{\|\mathcal{A}\|}Z$ and thus $Z = \|\mathcal{A}\|\tilde{Z} = \EE[K]$ is the expected value of the random linear image $K = \|\mathcal{A}\|\tilde{\mathcal{A}}(B_{\infty}^m)$ of the operator norm ball.

For the opposite implication, let $Z$ be a convex body in $\RR^d$ with the support function $$h_Z(u) = \EE\left[\left\|\sum_{i=1}^d u_i A_i \right\|_1\right]$$
for some $d$-tuple of random $m \times m$ Hermitian matrices $(A_1, \ldots, A_d)$. Consider a collection $\{(A_1^{(N)}, \ldots A_d^{(N)})\}_{N \in \mathbb{N}}$ of i.i.d. copies of $(A_1, \ldots, A_d)$, and define $Z_N$ as the random convex body with support function 
\[h_{Z_N}(u) := \frac{1}{N} \sum_{k=1}^N \left\|\sum_{i=1}^d u_iA_i^{(k)}\right\|_1.\]
Note that for all $N$, $Z_N$ is an $m$-operatope almost surely. By the strong law of large numbers for convex bodies \cite{Artstein_Vitale_1975}, $Z_N \to Z$ almost surely.
%$d$-tuple of random matrices with finite support such that $\EE[\|\mathcal{A}^{(\ee)} - \mathcal{A}\|] < \ee$ for the corresponding linear images and define $Z_{\ee}$ to be the convex body with support function
%\begin{align*}
%    h_{Z_{\ee}}(u) = \EE\left[\|\sum_{i=1}^d u_i A^{(\ee)}_i \|_1\right].
%\end{align*}
%Then, $Z_{\ee}$ is an $m$-operatope and 
%\begin{align*}
%    |h_Z(u) - h_{Z_{\ee}}(u)| \to 0 \text{ as } \ee \to 0.
%\end{align*}
Thus, there exists a sequence of $m$-operatopes that converge in the Hausdorff metric to $Z$ and hence $Z$ is an $m$-operanoid.
\end{proof}

Continuing with the illustration in \eqref{eq:2-operatope}, this result tells us that any $2$-operanoid specified by $2 \times 2$ real symmetric matrices must have a support function of the form:
\begin{align*}
   h(u) = \mathbb{E}\left[\max\left\{\left|\langle u, a + c \rangle\right| , \sqrt{\langle u, a - c \rangle^2 + \langle u, 2b \rangle^2} \right\}\right]
\end{align*}
for some \emph{random vectors} $a$, $b$ and $c$ in $\R^d$.

%\VC{This would be a good place to mention the structure of $2$-operanoids, but let's discuss.}

As our final equivalent characterization of operanoids, we recall that zonoids carry a functional-analytic characterization related to subspaces of $L^1([0,1], \RR)$; see Theorem 6.1 of \cite{bolker_class_1969}. The following result generalizes this perspective to $m$-operanoids. Let $L^1([0,1], \mathbb{H}^m)$ denote the linear space of compact operators $A: [0,1] \to \mathbb{H}^m$ such that the $L^1$ norm defined by
$$\|A\|_{L^1} := \int_{[0,1]} \mathrm{tr}(A(\omega)^*A(\omega))^{1/2} \dint \omega = \int_{[0,1]} \|A(\omega)\|_1 \dint \omega$$
is finite.

\begin{theorem} \label{t:L1-subspace}
A centrally symmetric convex body $Z \subset \RR^d$ centered at the origin is an m-\emph{operanoid} if and only if the normed space $(\RR^d, h_Z(\cdot))$ is isometric to a subspace of $L^1([0,1], \mathbb{H}^m)$.  %That is, its support function is of the form
%\begin{align*}
%    h_K(u) = \int_0^1 \| \sum_{i=1}^d u_if_i(t) \|_{1} \dint t,
%\end{align*}
%where for each $i$, $f_i : [0,1] \to \mathbb{M}_m(\RR)$ is an element of $\mathcal{L}^1([0,1], \mathbb{M}_m(\mathbb{R}))$.
\end{theorem}

\begin{proof}
    We follow closely the proof of Theorem 6.1 in \cite{bolker_class_1969}. Assume $Z$ is an $m$-operanoid for a $d$-tuple of random matrices $A_1, \ldots, A_d$ defined on a probability space $(\Omega, \mathcal{F}, P)$.
    Consider the map $T: \RR^d \to L^1(\Omega, \mathbb{H}^m)$ where for $x \in \RR^d$, $Tx : \omega \to \mathbb{H}^m$ is the function $\sum_{i=1}^d x_i A_i(\cdot)$. Then by Theorem \ref{t:zonoid_vitale_characterization}, $h_Z(x) = \int_{\Omega} \sum_{i=1}^d x_i A_i(\omega) \dint P(\omega) = \|Tx\|_{L^1}$, and thus $T$ is an isometry from $(\RR^d, h_Z(\cdot))$ to $L^1(\Omega, \mathbb{H}^m)$. As argued in Theorem 6.1 of \cite{bolker_class_1969}, the result follows from the fact that every separable subspace of $L^1(\Omega, \mathbb{H}^m)$ (for any probability space $\Omega$) is isometric to a subspace of $L^1([0,1], \mathbb{H}^m)$, and we have a finite-dimensional subspace which is clearly separable.
    % \VC{This last sentence is a bit confusing.  In particular, I don't understand the phrase ``the general probability space can be removed''.} \EO{I edited it to remove this phrase - is it ok now?} \VC{Can we discuss this?}

    Now assume there exists an isometry $T:(\RR^d, \| \cdot \|) \to L^1([0,1], \mathbb{H}^m)$ for a norm $\|\cdot\|$ on $\RR^d$. Let $e_1, \ldots, e_d$ be the standard basis in $\RR^d$ and set $A_i = Te_i$ for each $i$. Then, for each $x \in \RR^d$,
    \begin{align*}
        \|x\| = \left\|Tx \right\|_{L^1} = \left\|\sum_{i=1}^d x_i Te_i \right\|_{L^1} = \int_{[0,1]} \left\|\sum_{i=1}^d x_i A_i(\omega)\right\|_1 \dint \omega.
    \end{align*}
    Thus, $\|\cdot\| = h(Z, \cdot)$ for an operanoid $Z$ associated to the random matrices $A_1, \ldots, A_d$.
    % \VC{Adjust notation.} \EO{I've defined the $\mathcal{L}^1$ norm before the Theorem now - is this ok?} \VC{Looks good!}
 \end{proof}

This functional-analytic perspective yields several examples and non-examples of zonoids and operanoids.  A result due to Lindenstrauss \cite[Corollary 2]{lindenstrauss1964extension}
%\VC{Give a citation here?} \EO{Done}
states that all two-dimensional Banach spaces are isometrically embeddable into $L^1([0,1], \mathbb{R})$, which implies the well-known fact that all centrally symmetric convex bodies in $\R^2$ are zonoids and, in turn, $m$-operanoids for all $m$.  More broadly, we will see in Section~\ref{sec:properties} that the preceding result facilitates examples of $2$-operanoids (in fact, $2$-operatopes) in $\R^2$ that are not zonoids.

% \VC{What would be a good example here that illustrates the functional-analytic viewpoint?  Should we state the Lindenstrauss result that all centrally symmetric convex bodies in $\R^2$ are isometric to subspaces to $\mathcal{L}_1$?}

% \EO{Give example of utility of this perspective for zonoids with Lindenstrauss result that all centrally symmetric convex bodies in $\R^2$ are isometric to subspaces to $\mathcal{L}_1$ and connect to the work of Heinavaara which is from the functional analytic perspective.}

% \VC{Give forward pointer to later subsection where this is used.}

\subsection{More Examples}\label{sec:rm-examples}
In this section, we describe examples of operanoids corresponding to $d$-tuples of well-known random matrix ensembles.

\subsubsection{Wigner random matrices}\label{sec:wigner}

A Wigner random matrix $A \in \mathbb{H}^m$ is a random Hermitian matrix with diagonal entries $\{a_{ii}\}_{i=1}^m$ that are i.i.d. zero mean real-valued random variables and off-diagonal entries $\{a_{ij}\}_{1 \leq i < j \leq m}$ that form another family of i.i.d. zero mean complex-valued random variables with finite second moment that are independent of the diagonals.
Let $A_1, \ldots, A_d$ be Wigner random matrices with entries $(A_i)_{jk}$ for $i = 1, \ldots, d$ and $j,k \in [m]$.
For all $j, k \in [m]$, define the random vectors $a^{(jk)} = ((A_1)_{jk}, \ldots, (A_d)_{jk}) \in \CC^d$. Then, the linear combination $\sum_{i=1}^d u_i A_i$ is a Wigner ensemble with $jk$-th entry $\langle u, a^{(jk)}\rangle$. 

%Since the entries are linear combinations of random variables, a nice class of Wigner matrices to consider is those where the entries are independent and have stable distributions. That is, assume the vector $a$ has iid coordinates that have a stable distribution, such as Gaussian, Cauchy, or L\'{e}vy. Then the entries of $\sum_{i=1}^d u_iA_i$ are iid random variables with a distribution from the same class now with parameters depending on $u$, that is, $\sum_{i=1}^d u_iA_i$ is a Wigner ensemble.

Now, consider the setting where the vectors $a^{(jk)}$ defined above are real-valued Gaussian vectors $a^{(jk)} \sim \mathcal{N}(0, \Sigma)$ for all $j < k$ and $a^{(jj)} \sim \mathcal{N}(0, 2\Sigma)$. Then, 
\[\langle u, a^{(jk)} \rangle \sim \mathcal{N}(0, u^T\Sigma u), \, j < k, \qquad \langle u, a^{(jj)} \rangle \sim \mathcal{N}(0, 2u^T\Sigma u),\]
and $\sum_{i=1}^d u_i A_i$ has the distribution of a Gaussian Orthogonal Ensemble (GOE) $O$ multiplied by the constant $u^T\Sigma u$. The support function of the $m$-operanoid $Z$ associated to $A_1, \ldots, A_d$ is then
\begin{align*}
   h_{Z}(u) = \EE\left[\left\|\sum_{i=1}^d u_iA_i \right\|_1\right] =  (u^T\Sigma u)^{1/2}\EE\left[ \left\| O \right\|_1\right],
\end{align*}
and thus $Z$ is an ellipsoid. In particular, $m$-operanoids for all $m \geq 1$ include ellipsoids. This is already known since zonoids $(m=1)$ include ellipsoids.
%\item Suppose $a \sim \mathcal{MC}(\mu, \Sigma)$, i.e. suppose $a$ has a multivariate Cauchy distribution. Recall that $a$ then has the same distribution as the vector $Y^{-1/2}z$, where $z \sim \mathcal{N}(0,I)$ and $Y$  has a chi-squared distribution. Then $\langle u,a \rangle$ has a Cauchy distribution.

\subsubsection{Uniform permutation matrices}\label{sec:permutation}

Permutation matrices are square matrices with a single `1' in each column and row, and `0's elsewhere. The collection of $m \times m$ permutation matrices enumerates all $m!$ permutations of the coordinates of an $m$-dimensional vector.  A uniform permutation matrix is a random matrix chosen uniformly from this collection. Permutation matrices are not symmetric in general, but using the Jordan-Weilandt matrix construction \cite{stewart1990matrix}, an $m \times m$ permutation matrix $A$ can be lifted to a $2m \times 2m$ symmetric matrix $$\tilde{A} := \begin{bmatrix} 0 & A \\ A^\star & 0 \end{bmatrix}$$ such that $\|A\|_1 = \frac{1}{2}\|\tilde{A}\|_1$. We are thus interested in the $2m$-operanoid specified by $m \times m$ uniform permutation matrices lifted to symmetric matrices via this construction. We do not have a closed form expression for the support function of these operanoids, but we can numerically approximate it by sampling random $2m$-operatopes specified by samples of $m \times m$ uniformly random permutation matrices for a large number of summands $N$. Figure \ref{fig:perm_operanoid} plots such a sample for $N = 150$ and $m = 3, 6$ and $10$.

\begin{figure}[h]
     \centering
     \begin{subfigure}[b]{0.25\textwidth}
         \centering
         \includegraphics[width=\textwidth]{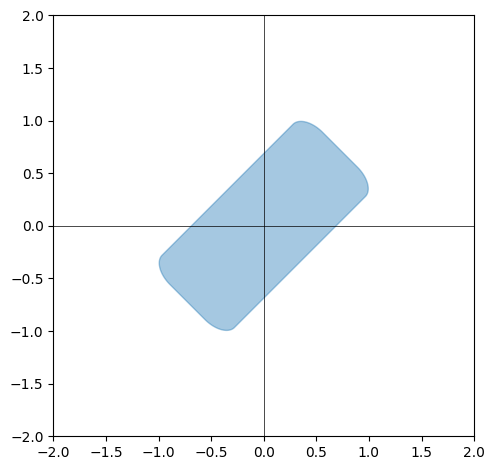}
         \caption{$m = 3$}
         \label{fig:3operanoid_perm}
     \end{subfigure}
     \hfill
     \begin{subfigure}[b]{0.25\textwidth}
         \centering
         \includegraphics[width=\textwidth]{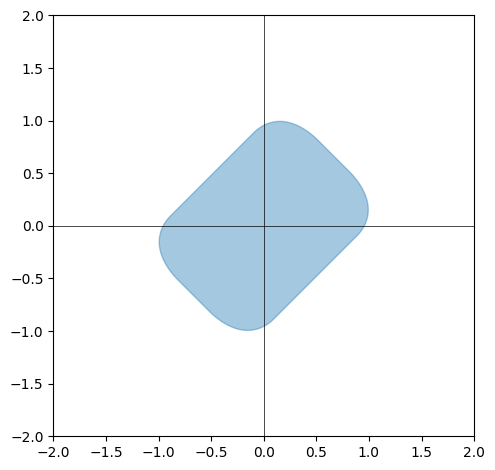}
         \caption{$m = 6$}
         \label{fig:6operanoid_perm}
     \end{subfigure}
     \hfill
     \begin{subfigure}[b]{0.25\textwidth}
         \centering
         \includegraphics[width=\textwidth]{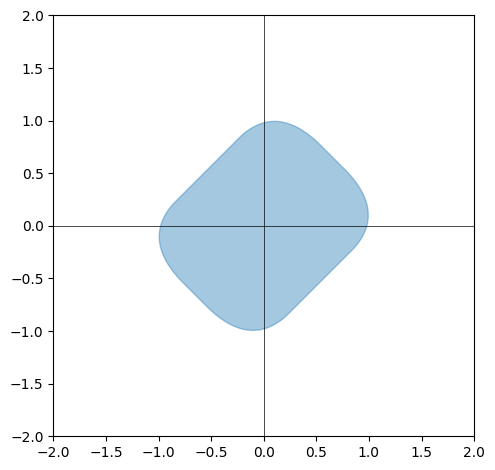}
         \caption{$m = 10$}
         \label{fig:10operanoid_perm}
     \end{subfigure}
        \caption{Samples of $6, 12$ and $20$-operatopes specified by i.i.d. uniformly random permutation matrices with $N = 150$ summands.}
        \label{fig:perm_operanoid}
\end{figure}

% \EO{Comment here on how to obtain these operanoids with symmetric matrices} \VC{Is this what you did above?}

\subsubsection{Wishart random matrices}\label{sec:wishart}

The Wishart distribution is a probability distribution over positive semidefinite matrices described as follows. For $\kappa > 0$ and  $m \in \mathbb{N}$, let $X$ denote a random  $\kappa m \times m$ matrix with i.i.d. entries that have zero mean and finite variance $\sigma^2$. Then, the $\kappa m \times \kappa m$ matrix $Y = \frac{1}{m\sqrt{\kappa}}XX^\star$ is a Wishart random matrix.  This ensemble is also known as a sample covariance matrix. 
As in the case of uniform permutation matrices, we do not have a closed form expression in general for support functions of the $\kappa m$-operanoids associated to a $d$-tuple of independent Wishart random matrices, but we can obtain an approximation by sampling the associated $\kappa m$-operatopes for a large number of summands $N$, see Figure \ref{fig:wishart_operanoid}.

\begin{figure}[h]
     \centering
     \begin{subfigure}[b]{0.25\textwidth}
         \centering
         \includegraphics[width=\textwidth]{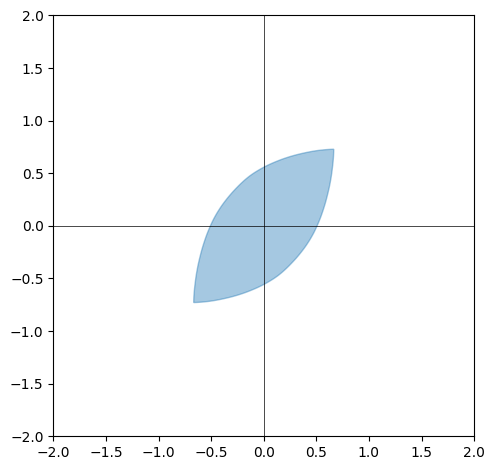}
         \caption{$m = 6, \kappa = 1/2$}
         \label{fig:3operanoid_wishart}
     \end{subfigure}
     \hfill
     \begin{subfigure}[b]{0.25\textwidth}
         \centering
         \includegraphics[width=\textwidth]{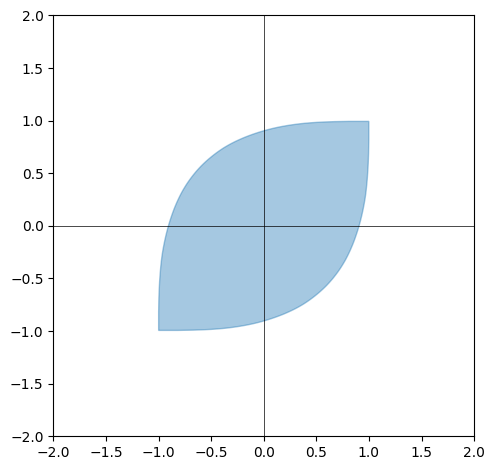}
         \caption{$m = 6, \kappa = 1$}
         \label{fig:6operanoid_wishart}
     \end{subfigure}
     \hfill
     \begin{subfigure}[b]{0.25\textwidth}
         \centering
         \includegraphics[width=\textwidth]{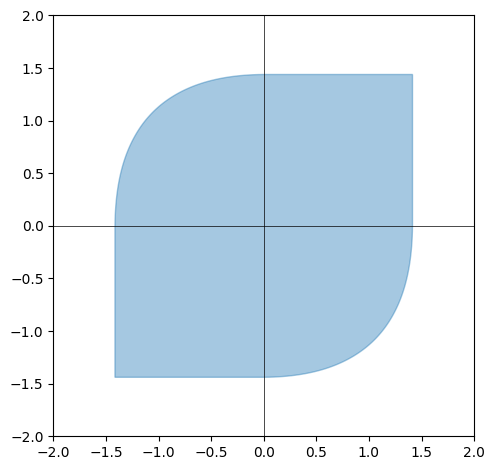}
         \caption{$m = 6, \kappa = 2$}
         \label{fig:10operanoid_wishart}
     \end{subfigure}
        \caption{Samples of $\kappa m$-operatopes with $N = 150$ summands associated to i.i.d. Wishart matrices with Gaussian entries where $m = 6$ and $\kappa = 1/2, 1, 2$ from left to right.
        % \VC{Should we delete this last sentence?} \EO{Yes, its deleted}
        }
        \label{fig:wishart_operanoid}
\end{figure}

\section{Properties of Operanoids} \label{sec:properties}
In this section we investigate various geometric and analytic aspects of operanoids.  Section \ref{sec:faces} discusses the facial structure of operatopes and operanoids, while Section \ref{sec:relation} presents a comparison with respect to other families of convex bodies.  Finally, Section \ref{sec:approximation} generalizes a result of \cite{bourgain_approximation_1989} on approximations of zonoids by zonotopes to prove rates on how well operatopes can approximate operanoids.

%%%%%%%%%%%%%%%%%%%%%%%%%%%%%%%%%%%%%%%%%%%%%%%%%%%%%%%%%%

\subsection{Facial Structure}\label{sec:faces}

%%%%%%%%%%%%%%%%%%%%%%%%%%%%%%%%%%%%%%%%%%%%%%%%%%%%%%%%%%

In this subsection, we investigate the facial structure of operanoids.  An \emph{exposed face} of a convex set $K$ is a subset $F \subseteq K$ such that $F$ is equal to the intersection of $K$ and a hyperplane that supports $K$; equivalently, $F$ is equal to the set of maximizers of a linear functional over $K$.  We recall here a key observation pertaining to exposed faces of a convex set $K$ that is expressed as the image under a linear map $\mathcal{L}$ of a convex set $\tilde{K}$ -- the preimage under $\mathcal{L}$ of an exposed face $F$ of $K$ is an exposed face of $\tilde{K}$, i.e., $\mathcal{L}^{-1}(F) \cap \tilde{K}$ is an exposed face of $\tilde{K}$.

We begin by reviewing some results from the literature on zonoids.  First, we observe that as a zonotope is a linear image of a $\infty$-norm ball and the faces of an $\infty$-norm ball are themselves (shifted) $\infty$-norm balls, all the faces of a zonotope are also zonotopes .  More broadly, all the exposed faces of a zonoid are also zonoids \cite[Theorem 3.2]{bolker_class_1969}.
% \VC{Citation to Bolker for this latter result?} \EO{Done.}
To make progress on the facial structure of operanoids, we state the following result on faces of an operator norm ball \cite[Theorem 4]{So_1990}.

% \VC{Do we want to say that any polytope with centrally-symmetric two-dimensional faces must be a zonotope \cite[Theorem 3.5.2]{Schneider_2013}, or that any polytope that is also a zonoid must be a zonotope \cite{Schneider_2013}?  These are nice, but not things we are able to extend directly.}

% \EO{Check the results below can be written for operator norm ball in $\mathbb{H}^m$} \VC{I will look into this.}

\begin{lemma}\label{thm:opball_faces}
A subset $F \subseteq B_{\infty}(\mathbb{H}^m)$ of an operator norm ball is a closed proper (non-empty) exposed face of $B_{\infty}(\mathbb{H}^m)$ if and only if there exists $r \in \{1,\dots,m - 1\}$
% \VC{Should this be $m-1$ instead of $m$, since we're talking proper faces?} \EO{Yes, I updated this.}
and unitary matrices $U$ and $V$ such that $$F = \left\{U\begin{bmatrix} I_r & 0 \\ 0 & A \end{bmatrix} V : A \in B_{\infty}(\mathbb{H}^{m-r})\right\}.$$.
\end{lemma}

In words, the exposed faces of an operator norm ball are shifted operator norm balls, very much in analogy to the situation with $\infty$-norm balls.  Based on this characterization, we prove next a result about the faces of operanoids.

\begin{proposition}\label{prop:operanoid_faces}
    Each exposed face of an $m$-operanoid is an $m$-operanoid.  Further, each exposed face of an $m$-operatope is an $m$-operatope.
\end{proposition}

\begin{proof}
We begin by making an observation about preimages of exposed faces of a linear image of an operator norm ball.  Let $F \subseteq \mathcal{L}(B_{\infty}(\mathbb{H}^m)) \subseteq \mathbb{R}^d$ be a face of the image of the operator norm ball under a linear map $\mathcal{L}$.  From our discussion above, the preimage $\mathcal{L}^{-1}(F)$ of $F$ must be an exposed face of $B_{\infty}(\mathbb{H}^m)$. By Lemma \ref{thm:opball_faces}, there exists $r \in \{1,\dots, m\}$ and unitary matrices $U$ and $V$ such that %All $k$-faces of an $m$-operatope of order $N = 1$ are linear images of a face of $B_{\infty}^m$ and have the form
\begin{align*}
    \mathcal{L}^{-1}(F) = \left\{U\begin{bmatrix} I_r & 0 \\ 0 & A \end{bmatrix} V : A \in B_{\infty}(\mathbb{H}^{m-r})\right\}.
\end{align*}
Thus, there are matrices $B_1, \ldots, B_d \in \mathbb{H}^{m}$ such that
\begin{align}\label{e:face_supp}
    &F =  \left\{ \left(\langle B_1, U\begin{bmatrix} I_{r} & 0 \\ 0 & A \end{bmatrix}V  \rangle, \ldots, \langle B_d, U\begin{bmatrix} I_{r} & 0 \\ 0 & A \end{bmatrix}V  \rangle \right) : A \in B_{\infty}(\mathbb{H}^{m-r})\right\} \nonumber\\
    &=  \left\{ \left(\langle U^{\star}B_1V^{\star}, \begin{bmatrix} I_{r} & 0 \\ 0 & A \end{bmatrix}  \rangle, \ldots, \langle U^{\star}B_dV^{\star}, \begin{bmatrix} I_{r} & 0 \\ 0 & A \end{bmatrix}  \rangle \right) : A \in B_{\infty}(\mathbb{H}^{m-r})\right\} \nonumber\\
    &=\left\{ \left(\mathrm{Tr}([U^{\star}B_1V^{\star}]_{r} + \langle [U^{\star}B_1V^{\star}]_{m-r}, A \rangle, \ldots, \mathrm{Tr}([U^{\star}B_dV^{\star}]_{r}) + \langle [U^{\star}B_dV^{\star}]_{m-r}, A \rangle \right) : A \in B_{\infty}(\mathbb{H}^{m-r})\right\}
    %\left(\mathrm{Tr}([U^{\star}B_1V^{\star}]_{r}), \ldots, \mathrm{Tr}([U^{\star}B_dV^{\star}]_{r})\right) + \left\{ \left(\langle [U^{\star}B_1V^{\star}]_{m-r}, A \rangle, \ldots, \langle [U^{\star}B_dV^{\star}]_{m-r}, A \rangle \right) : A \in B_{\infty}(\mathbb{H}^{m-r})\right\},
\end{align}
where $[A]_{r}$ for $A \in \mathbb{H}^{m}$ is the $r \times r$ principal submatrix of $A$ corresponding to rows/columns $\{1, \ldots, r\}$ of $A$ and $[A]_{m-r}$ is the $(m-r) \times (m-r)$ principal submatrix corresponding to rows/columns $\{r+1, \ldots, m\}$.
% \VC{Should this last bit be $\{r+1,\dots,m\}$?} \EO{Yes, fixed.}

Now let $F_{\xi}$ be an exposed face of an $m$-operanoid $K$. By Theorem 1.7.2 in \cite{schneider_2013},
\begin{equation}\label{eq:face-limit}
    \begin{aligned}
        h(F_{\xi}, u) &= \lim_{\lambda \to 0^+} \frac{\EE\left[\left\|\sum_{i=1}^d (\xi_i + \lambda u_i)A_i \right\|_1 \right] - \EE\left[\left\|\sum_{i=1}^d \xi_i A_i \right\|_1\right]}{\lambda} \\
        &=  \EE\left[ \lim_{\lambda \to 0^+} \frac{\left\|\sum_{i=1}^d (\xi_i + \lambda u_i)A_i \right\|_1 - \left\|\sum_{i=1}^d \xi_i A_i \right\|_1}{\lambda} \right] = \EE[h(F_{\xi,A}, u)],
    \end{aligned}
\end{equation}
where $F_{\xi,A}$ is the face of a linear image $\mathcal{L}_A$ of $B_{\infty}(\mathbb{H}^m)$. Above we have used the Dominated Convergence Theorem to switch the limit and expectation. By Theorem \ref{thm:opball_faces} and \eqref{e:face_supp}, there exist unitary $U$ and $V$ such that 
\begin{align*}
    h(F_{\xi, A},u) = \langle u, \left(\mathrm{Tr}([U^\star A_1 V^\star]_{r}), \ldots, \mathrm{Tr}([U^\star A_d V^\star]_{r})\right) \rangle + \left\|\sum_{i=1}^d u_i [U^\star A_i V^\star]_{m-r}\right\|_1.
\end{align*}
Thus $F_{\xi}$ is the expected value of $F_{\xi,A}$, and 
\begin{align*}
 h(F_{\xi}, u)  =   \left\langle u, \left(\EE\left[\mathrm{Tr}([U^\star A_1 V^\star]_{r})\right], \ldots, \EE\left[\mathrm{Tr}([U^\star A_d V^\star]_{r})\right] \right) \right\rangle + \EE\left[\left\|\sum_{i=1}^d u_i [U^\star A_iV^\star]_{m-r}\right\|_1\right], 
\end{align*}
showing the $F_{\xi}$ is a translated $(m-r)$-operanoid, and is therefore an $m$-operanoid.

Finally, in the preceding paragraph, if $K$ is an operatope rather than a general operanoid, then the expectations in the first line \eqref{eq:face-limit} are finite sums.  Therefore, all subsequent expectations are also given by finite sums, and we can conclude that each exposed faces of an $m$-operatope is also an $m$-operatope.
\end{proof}

% \VC{Do we want to non-examples of operanoids right here to illustrate the utility of this result?  Of should we defer this to later?}

We illustrate one consequence of this result.  It is well-known that if a polytope is a zonoid, then it must be a zonotope \cite{bolker_class_1969}.  Based on the above proposition, we extend this observation to operanoids.  Specifically, if a polytope is an $m$-operanoid for any $m$, then all its faces must also be $m$-operanoids and hence centrally-symmetric.  Therefore, this polytope must be a zonotope, because we have from \cite[Theorem 3.5.1]{schneider_2013} that a polytope with centrally-symmetric two-dimensional faces must be a zonotope.

\subsection{Relation between Operanoids and Other Families of Convex Bodies}\label{sec:relation}

In this subsection, we investigate the relationship between operatopes/operanoids and other families of convex bodies.  As a first remark, the examples from the preceding subsections (e.g., ellipsoids and their Minkowski sums) imply that there clearly exist operatopes that are not zonotopes.  A natural next question is whether all zonoids are $m$-operatopes for some $m$.  It is instructive to consider this question in $\R^2$, as all centrally symmetric convex bodies in $\R^2$ are zonoids \cite[Theorem 5.4]{bolker_class_1969}.  However, $m$-operatopes must be semialgebraic as they are expressed as linear images of a semialgebraic convex body.  As there are clearly centrally symmetric convex bodies in $\R^2$ that are not semialgebraic, e.g., $\ell_p$ balls for $p$ irrational, there are zonoids that are not $m$-operatopes for any $m$.  Beyond this basic observation, one source of difficulty is that the question of whether a particular convex body is a zonoid is challenging, and it is known as the \emph{zonoid problem} \cite{Bolker_Zonoid_Problem}.  This difficulty is also inherited by operanoids.  Nevertheless, we leverage the characterizations from the preceding subsections to present three sets of observations on the relationship between operanoids and zonoids and between operanoids and general families of convex bodies.

\paragraph{Operanoids in $\R^2$}  As all $m$-operanoids in $\R^2$ are zonoids, it is natural to ask how to express the support function of $m$-operanoids in $\RR^2$ in the form of the support function of a zonoid.  For an $m$-operatope $Z$ centered at the origin with one summand, we have that
$$h_Z(u) = \left\|u_1A_1 + u_2A_2 \right\|_1.$$
As $Z$ is also a zonoid, it must be the case that $h_Z(u) = \int_{\RR^2} |\langle u,x \rangle| \dint \mu(x)$ for some measure $\mu$ that depends on $A_1, A_2$.  An explicit expression for such a $\mu$ is given by the \emph{tracial joint spectral measure} of $A_1$ and $A_2$ as defined in \cite{heinavaara2025tracial}.  The same work provides a characterization of the singular and continuous parts of the joint spectral measure.  To obtain an analogous zonoidal description of an $m$-operanoid $Z$ specified now in terms of a pair of random matrices $A_1, A_2$, the corresponding measure $\mu$ is the mean measure $\Lambda$ of a \emph{random} tracial joint spectral measure associated to $A_1, A_2$:
$$h_Z(u) = \EE\left[\left\|u_1A_1 + u_2A_2 \right\|_1\right] = \EE\left[\int_{\RR^2} |\langle u,x \rangle| \dint \mu(x)\right] = \int_{\RR^2} |\langle u,x \rangle| \dint \Lambda(x).$$
% \VC{Can we say a bit more explicitly what $\mu$ and $\Lambda$ is in terms of the random $A_1, A_2$?}
% \EO{The precise form is given in Heinevaara's paper - would just a more explicit reference suffice? It's a bit complicated to copy here I think, but we could.}

\paragraph{Operanoids and zonoids in higher dimensions} For dimensions $d \geq 3$, zonoids form a strict subclass of convex bodies and the question we consider is whether there are any $m$-operanoids that are not zonoids in $\R^3$.  We answer this question in the affirmative by appealing to the recent work \cite{HEINAVAARA2024} which proved the following result.
%\EO{We first recall that $S^1$ denotes the Schatten-1 trace-class, the set of compact operators on a complex separable Hilbert space. }

\begin{theorem}[Theorem 1.0.3 in \cite{HEINAVAARA2024}]
%Let $S^1$ be the Schatten-1 trace class of operators on some Hilbert space. 
The space of real symmetric $2 \times 2$ matrices endowed with the trace norm is \emph{not} linearly isometric to any subspace of $L^1([0,1],\RR)$.
\end{theorem}

% in two ways.  First, it is shown that that are subspaces of the trace class of operators on a Hilbert space that are not linearly isometric to $L^1([0,1],\mathbb{R})$.  Second, the above result was proved 

The results in \cite{HEINAVAARA2024} are actually stated more generally for any $L^p$ such that $p \in [1, \infty)$, $p \neq 2$, but the preceding statement with $p=1$ is all we need for our arguments.  This theorem implies that there is a $2$-operanoid in $\RR^3$ that is \emph{not} a zonoid, based on Theorem~\ref{t:L1-subspace}.  In particular, we can extract from \cite{HEINAVAARA2024} the conclusion that the $2$-operatope with support function
\begin{align*}
    h(u) = \left\| \begin{bmatrix} u_3 + u_1 & u_2 \\ u_2 & u_3 - u_1 \end{bmatrix} \right\|_1,
\end{align*}
is not a zonoid (see Figure \ref{fig:nonzonoid}).  By the computation in \eqref{eq:2-operatope}, this is the support function of the convex hull of a $2$-dimensional disc and a perpendicular line segment.
%This convex body is linearly isomorphic to the $2 \times 2$ operator norm ball over symmetric matrices.

\begin{figure}[h]
    \centering    \includegraphics[width=0.25\linewidth]{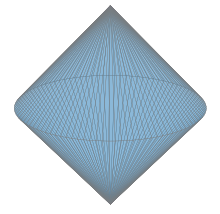}
    %\vspace{-0.5in}
    \caption{A 3-dimensional 2-operatope that is not a zonoid.}
    \label{fig:nonzonoid}
\end{figure}

\paragraph{Convex bodies that are not operanoids} For dimensions $d \geq 3$ zonoids form a strict subclass of convex bodies, motivating the study of this class of convex bodies.  We show next that there exist centrally symmetric convex bodies in $\RR^d$ for $d \geq 3$ that are not $m$-operanoids.  Combined with the above observation that there exist $m$-operanoids that are not zonoids in $\R^3$ for any $m \geq 2$, our results provide a justification for studying operanoids as a structured class of convex bodies.  For explicit examples of convex bodies that are not operanoids, consider a polytope in $\RR^d$ for $d \geq 3$ that is not a zonotope, such as the $\ell_1$ unit ball. By Theorem 3.5.2 in \cite{schneider_2013}, the $\ell_1$ unit ball is not a zonoid because its $2$-dimensional faces are not centrally symmetric. The same reasoning implies that the $\ell_1$ unit ball is not an $m$-operanoid for any $m$ by Proposition \ref{prop:operanoid_faces}.  Building on this observation, we note that the nuclear norm ball of $n \times n$ matrices for any $n \geq 3$ is not an $m$-operanoid for any $m$, as the $\ell_1$ ball in $\R^n$ may be expressed as a linear image of an $n \times n$ nuclear norm ball by projecting onto the diagonal entries (by contradiction, if the nuclear norm ball were an $m$-operanoid, then any linear image of it would also be an $m$-operanoid).  

Finally, the class of $m$-operanoids must lie in the complement of the set of simplicial polytopes, since the faces of $m$-operatopes are centrally symmetric and $m$-operanoids that are not $m$-operatopes must be decomposable (and hence are not simplicial polytopes which are indecomposable, see Corollary 3.2.13 in \cite{schneider_2013}). Thus, since simplicial polytopes form a dense subset of $\mathcal{K}^d$, the set of $m$-operanoids must be nowhere dense. We state this conclusion in the following proposition.
\begin{proposition}
    For $d \geq 3$, the class of $m$-operanoids in $\RR^d$ are nowhere dense in the class of centrally symmetric convex bodies in $\RR^d$. %If $m^2 \leq d-1$, $m$-operanoids are a nowhere dense in the set of centrally symmetric convex bodies in $\RR^d$.
\end{proposition}

\subsection{Approximation of Operanoids by Operatopes}\label{sec:approximation}

As our final result of this section, we consider the question of approximating $m$-operanoids by $m$-operatopes.  Specifically, how many summands $N$ are required in a Minkowski sum description of an $m$-operatope to approximate a given $m$-operanoid to a desired accuracy?  For the case $m=1$ of zonotope approximations of zonoids, this question was answered in \cite{bourgain_approximation_1989}, followed by improvements tightening the upper bound to a matching lower bound \cite{matouvsek1996improved, siegel_optimal_2023}.  Here we generalize some of this previous work to obtain a result for arbitrary $m$.

\begin{theorem}\label{t:approx_rate}
Let $Z$ be an $m$-operanoid centered at the origin and defined by a $d$-tuple of random matrices $(A_1, \ldots, A_d)$ in $\mathbb{H}^m$ with entries bounded by $L$. There exists a sequence of $m$-operatopes $\{Z_N\}_{N \geq 1}$ -- each $Z_N$ with $N$ summands -- such that for all $\ee > 0$ and an absolute constant $c > 0$, we have that
\begin{align*}
    N \geq cm^2L^2d^2 \ee^{-2} ~~~ \Rightarrow ~~~ d_H(Z, Z_N) \leq \ee.
\end{align*}
% \VC{I restated this result slightly -- please let me know if this seems fine.}\EO{Yes, sounds good.}
% implies that the approximation with respect to the Hausdorff metric satisfies
% \begin{align*}
%     d_H(Z, Z_N) \leq \ee.
% \end{align*}
\end{theorem}

\begin{proof}
    We will prove this result using the probabilistic method.  In particular, we will construct random $m$-operatopes with $N$ summands and prove that they exhibit the desired approximation quality with high probability.  This in turn yields the existence of such operatope approximations. Throughout the proof, absolute constants appearing in upper bounds will be denoted $c_i$ for $i \in \mathbb{N}$.
    
    For $N \geq 1$ and $j =1, \ldots, N$, let $(A_1^{(j)}, \ldots, A_d^{(j)})$ be i.i.d. copies of the $d$-tuple of random matrices $(A_1, \ldots, A_d)$. Then for all $N \geq 1$, define the random $m$-operatope $Z_N$ with support function
    \begin{align*}
        h_{Z_N}(u) = \frac{1}{N} \sum_{j=1}^N \left\|\sum_{i=1}^d u_i A_i^{(j)} \right\|_1.
    \end{align*}
    %The goal the to show that with positive probability, for $N$ large enough,
    %\begin{align*}
    %    \PP(d_H(Z, Z_N) \leq \ee) \geq 
    %\end{align*}
    % \EO{For the $m$-operanoid $Z$ associated to the random matrices $(A_1, \ldots, A_d)$,}  \VC{Which $Z$ is this and from which theorem?}
    We observe that for $u,v \in \mathbb{S}^{d-1}$ and $t \geq 0$,
    \begin{align*}
        &\PP\left( |h_Z(u) - h_{Z_N}(u) - h_Z(v) + h_{Z_N}(v)| \geq t\right) \\
        &=  \PP\left( |h_Z(u) - h_Z(v) - (h_{Z_N}(u) - h_{Z_N}(v))| \geq t\right) \\
        &=  \PP\left( \left|\EE\left[\left\|\sum_{i=1}^d u_iA_i\right\|_1 - \left\|\sum_{i=1}^d v_i A_i \right\|_1 \right] - \frac{1}{N} \sum_{j=1}^N \left(\left\|\sum_{i=1}^d u_iA^{(j)}_i \right\|_1 - \left\|\sum_{i=1}^d v_i A^{(j)}_i \right\|_1\right)\right| \geq t\right).
    \end{align*}
    Now for each $j \in \{1, \ldots, N\}$, define $X_j := \left\|\sum_{i=1}^d u_iA^{(j)}_i \right\|_1 - \left\|\sum_{i=1}^d v_i A^{(j)}_i \right\|_1$. Note that the random variables $X_j$ satisfy the bounds
    \begin{align*}
       -\|u-v\|_2 mL\sqrt{d} \leq  -\|u-v\|_2\left(\sum_{i=1}^d \left\|A^{(j)}_i\right\|^2_1\right)^{\frac{1}{2}} \leq X_j  \leq  \|u-v\|_2\left(\sum_{i=1}^d \left\|A^{(j)}_i\right\|^2_1\right)^{\frac{1}{2}} \leq \|u-v\|_2 mL\sqrt{d}.
    \end{align*}
Hoeffding's inequality \cite[Theorem 2.2.6]{VershyninBook} then implies
\begin{align*}
    \PP\left(\frac{1}{N}\sum_{j=1}^N \left(X_j - \mathbb{E}[X_j] \right) \geq t\right) \leq \exp\left(- \frac{t^2N}{2m^2L^2d\|u-v\|_2^2}\right).
\end{align*}
Thus, the collection $\{h_Z(u) - h_{Z_N}(u)\}_{u \in \mathbb{S}^{d-1}}$ has sub-Gaussian increments with constant $\frac{mL\sqrt{d}}{\sqrt{N}}$. 

% \VC{The following sentence feels a bit confusing -- is it possible to rephrase it?  Also, it would be good to clarify that the $u_0$ is fixed but arbitrary.} \EO{Oof yea I agree, I've edited it - is it ok now?}

Next, let $u_0 \in \mathbb{S}^{d-1}$ be an arbitrary fixed unit vector. By the tail bound version of Dudley's integral inequality \cite[Theorem 8.1.6]{VershyninBook}, we have for all $s \geq 0$ that:
\begin{align*}
    \sup_{u \in \mathbb{S}^{d-1}} \left(|h_Z(u) - h_{Z_N}(u)| - |h_Z(u_0) - h_{Z_N}(u_0)|\right) 
    &\leq \frac{c_1 mL\sqrt{d}}{\sqrt{N}}\left[ \int_0^{2} \sqrt{\log \mathcal{N}(\mathbb{S}^{d-1}, \|\cdot\|_2, \delta)} \, \dint \delta + 2s \right]
\end{align*}
with probability at least $1 - 2e^{-s^2}$. Here, $\mathcal{N}(\mathbb{S}^{d-1}, \|\cdot\|_2, \delta)$ is the $\delta$-covering number of $\mathbb{S}^{d-1}$ with respect to the $\ell_2$ distance.
Recall from Corollary 4.2.13 and (4.10) in \cite{VershyninBook} that
\begin{align*}
  \mathcal{N}(\mathbb{S}^{d-1}, \|\cdot\|_2, \delta) \leq \left(\frac{3}{\delta}\right)^d,  
\end{align*}
which gives the upper bound
\begin{align*}
 \int_0^{2} \sqrt{\log \mathcal{N}(\mathbb{S}^{d-1}, \|\cdot\|_2, \delta)} \, \dint \delta \leq \int_0^{2} \sqrt{d \log \left(\frac{3}{\delta}\right)} \, \dint \delta \leq c_2\sqrt{d}  
\end{align*}
and thus for all $s \geq 0$,
\begin{align}\label{e:dudley_bnd}
    \sup_{u \in \mathbb{S}^{d-1}} \left(|h_Z(u) - h_{Z_N}(u)| - |h_Z(u_0) - h_{Z_N}(u_0)|\right) 
    &\leq \frac{c_3mL\sqrt{d}}{\sqrt{N}}\left(\sqrt{d} + s \right).
\end{align}
By another application of Hoeffding's inequality, we have for $t \geq 0$,
\begin{align*}
    \PP\left(|h_Z(u_0) - h_{Z_N}(u_0)| \geq t\right) \leq \exp\left(- \frac{2t^2N}{m^2L^2d}\right),
\end{align*}
and letting $t = \frac{mLd}{\sqrt{N}}$ gives
\begin{align*}
    \PP\left(|h_Z(u_0) - h_{Z_N}(u_0)| \geq \frac{mLd}{\sqrt{N}}\right) \leq \exp\left(- 2d\right).
\end{align*}
Finally, letting $s = \sqrt{d}$ in \eqref{e:dudley_bnd} and applying a union bound, we have
\begin{align*}
    d_H(Z, Z_N) &= \sup_{u \in \mathbb{S}^{d-1}} |h_Z(u) - h_{Z_N}(u)| \\
    &\leq |h_Z(u_0) - h_{Z_N}(u_0)| + \sup_{u \in \mathbb{S}^{d-1}} \left(|h_Z(u) - h_{Z_N}(u)| - |h_Z(u_0) - h_{Z_N}(u_0)|\right) \\
    &\leq \frac{c_4 mLd}{\sqrt{N}},
\end{align*}
with probability at least $1 - 2e^{-d} - e^{-2d}$.
Thus, there exists a sequence of $m$-operatopes $\{Z_N\}_{N \geq 1}$ such that for all $\ee > 0$, if
$N \geq c_5 m^2L^2d^2 \ee^{-2}$,
then $d_H(Z, Z_N) \leq \ee$.
\end{proof}

% \VC{Should we remark here about the connection to neural network approximations?  We don't do anything else with this result, correct?}

% \EO{Right, yea we could mention it here or in the discussion with an open question, and maybe describe the neural network architecture that corresponds to a $m$-operatope?}

% \VC{Mention linear optimization over operanoids.}

This result suggests some computational consequences, particularly in the context of optimization.  Suppose we wish to maximize a linear functional over an operanoid $Z$, i.e., evaluate the support function of $Z$ in a given direction $u$.  Based on Theorem \ref{t:approx_rate}, if we are able to sample matrices from the measure underlying the description of $Z$, then we can obtain an operatope approximation $\tilde{Z}$ so that $|h_{\tilde{Z}}(u) - h_Z(u)| \leq \varepsilon$ for a desired $\varepsilon$, with the upshot being that $h_{\tilde{Z}}(u)$ can be obtained via a spectral computation.

\section{Noncommutative Zonoids}\label{sec:NCzonoids-overview}
For a fixed $m$, the limit of $m$-operatopes as the number of summands $N$ grows can alternatively be viewed as taking the limit of $m N$-operatopes consisting of a single summand as $N$ grows, where the parameterizing matrix is restricted to be $m\times m$ block diagonal.  Relaxing this constraint leads to an asymptotic regime concerning the behavior of $m$-operanoids as the size $m$ of the matrices grows. This is a new setting that is not defined for zonoids, and a different perspective is required to investigate the associated limiting objects.  Under certain conditions on the matrices involved, the theory of noncommutative probability gives us the tools to make sense of such limits.  In particular, the general framework of noncommutative probability allows us to define a novel class of convex bodies called noncommutative zonoids characterized by noncommutative random variables.  We show that these objects contain $m$-operanoids as special cases, and moreover, they characterize certain limits of $m$-operanoids as $m$ grows.  Section \ref{sec:NCrandomvariables} provides some background on noncommutative random variables.  Section \ref{sec:NCzonoids} then formally presents the class of noncommutative zonoids.  This includes Section \ref{sec:freezonoids}  focusing on a subclass called free zonoids, which are convex bodies derived from free random variables.

\subsection{Noncommutative Random Variables}\label{sec:NCrandomvariables}

We briefly recall here the definition of a noncommutative probability space and additional relevant concepts.  For a more complete exposition, we refer the reader to \cite{Voiculescu_lectures}.

A $\star$-algebra $\mathcal{A}$ is an algebra over the field of complex numbers $\mathbb{C}$ equipped with an involution $\star$ such that $(ab)^{\star} = b^{\star}a^{\star}$ for all $a,b \in \mathcal{A}$. A $C^{\star}$-algebra is a $\star$-algebra equipped with a norm $\| \cdot\|: \mathcal{A} \to \RR_{+}$ such that
$(\mathcal{A}, \| \cdot\|)$ is a complete normed vector space,
$\|a^{\star}\| = \|a\|$, and $\|ab \| \leq \|a\|\|b\|$ for all $a,b \in \mathcal{A}$.
We then have the following definition of a noncommutative probability space.

\begin{definition}
   A \emph{noncommutative (NC) probability space} $({\mathcal A},\tau)$ consists of a $\star$-algebra $\mathcal A$ containing an identity element $\mathbf{1}$, together with a linear functional ${\tau: {\mathcal A} \rightarrow {\mathbb C}}$ such that $\tau(\mathbf{1}) = \mathbf{1}$ and $\tau(aa^{\star}) \geq 0$ for all $a \in \mathcal A$, i.e. $\tau$ maps positive elements of $\mathcal{A}$ to positive real numbers. An NC probability space $(\mathcal{A}, \tau)$ is a \emph{$C^{\star}$-probability space} if $\mathcal{A}$ is a $C^{\star}$-algebra. %\VC{Let's discuss this definition.}
\end{definition}

Another common choice of algebra to define an NC probability space is a von Neumann algebra. However, in the following, we will continue to focus on $C^\star$-probability spaces, as these allow for a continuous functional calculus that is needed for the objects studied in this section to be well-defined. For a $C^\star$-probability space $(\mathcal{A}, \tau)$, we will refer to an element $a \in \mathcal{A}$ as a NC random variable. %If we assume $a$ is normal, i.e. $a a^* = a^* a$, and let $\sigma(a)$ denote the spectrum of $a$, then the \emph{continuous functional calculus} allows us to make sense of the expression $f(a)$ as a bounded operator for a continuous function $f: \sigma(a) \to \RR$. \VC{This last sentence suggests a link to Hilbert spaces that might not be obvious at this stage -- let's discuss.}

\subsubsection{The distribution of a NC random variable}

We next describe how to understand the distribution of a NC random variable. Let $(\mathcal{A}, \tau)$ be a $C^{\star}$-probability space and $a \in \mathcal{A}$. Define $\mu_a: \mathbb{C}[x] \to \mathbb{C}$ as the linear functional on the polynomials (in the indeterminate $x$) given by
\[\mu_a(P) := \tau(P(a)).\]
We will refer to $\mu_a$ as the \emph{distribution} of $a$. More generally, the \emph{joint distribution} $\mu_{(a_1, \ldots, a_d)}$ of a collection of NC random variables $(a_1, \ldots, a_d)$ is the complex-valued linear map defined on the set of polynomials (in the indeterminates $(x_1, \ldots, x_d)$) given by
\[\mu_{(a_1, \ldots, a_d)}(P) = \tau(P(a_1, \ldots, a_d)).\]
The distribution $\mu_{(a_1, \ldots, a_d)}$ is thus completely determined by the noncommutative moments $$\tau(a_{i_1}\cdots a_{i_k}), \text{ for all } k \in \mathbb{N} \text{ and } i_1, \ldots, i_k \in \{1, \ldots d\}.$$

The \emph{Gelfand-Naimark-Segal Theorem} says that any $C^{\star}$-probability space $(\mathcal{A}, \tau)$ can be identified with a norm-closed ${\star}$-subalgebra $\mathcal{A} \subseteq B(\mathcal{H})$ of the collection of bounded operators on a Hilbert space $\mathcal{H}$, where for $a \in \mathcal{A}$, $\tau(a) = \langle a \xi, \xi\rangle$ and $\xi \in \mathcal{H}$ is a unit vector with $\|\xi\| = 1$.  
For a self-adjoint element $a = a^* \in (\mathcal{A}, \tau)$, this identification allows the distribution $\mu_a$ to be described by a probability measure on $\RR$. In particular, letting $\sigma(a)$ denote the spectrum of $a$ (realized as a self-adjoint bounded operator) and using the continuous functional calculus, the spectral theorem gives the following. There exists a compactly supported probability measure $\nu$ called the \emph{spectral distribution of $a$} on $\RR$ such that %we can write 
%$$\mu_a(P) = \int_{\RR}P(t) \dint \nu(t).$$ 
%Furthermore, letting $\sigma(a) \subset \RR$ denote the spectrum of $a$ (realized as a self-adjoint bounded operator), the continuous functional calculus extends this to continuous functions. That is, 
for any continuous function $f: \sigma(a) \to \RR$, 
\begin{align}\label{eq:spectalmeasure}
    \tau(f(a)) = \int_{\RR} f(t) \dint \nu(t).
\end{align}
We remark here again that the assumption of a $C^{\star}$-algebra for the probability space is motivated by access to the continuous functional calculus, which ensures $\eqref{eq:spectalmeasure}$ is well-defined for all continuous functions $f$, instead of just polynomials. Specifically, we would like to study the absolute value of a NC random variable in the sequel.
%Indeed, polynomials are dense in the space of continuous functions, and thus $$\int_{\RR} f(t) \dint \nu(t) = \lim_{n \to \infty} \int_{\RR} P_n(t) \dint \nu(t) = \lim_{n \to \infty} \tau(P_n(a)).$$

Another way to characterize the distribution of an NC random variable is through functional transforms. The most relevant of these in our setting is the \emph{$R$-transform}. To define it, we first recall that the \emph{Steiltjes} or \emph{Cauchy transform} $G$ of an NC random variable $a$ is given by
\begin{align*}
 G(z) := \frac{1}{z} +  \sum_{n=1}^{\infty} \frac{\tau(a^n)}{z^{n+1}},   
\end{align*}
and if $a$ is self-adjoint with spectral distribution $\nu$,
\begin{align*}
    G(z) = \tau\left(\frac{1}{z - a}\right) = \int_{\RR} \, \frac{1}{z - t} \, \dint \nu(t), \qquad z \in \mathbb{C}/\RR.
\end{align*}
\begin{definition} The \emph{$R$-transform} of an NC random variable $a$ is defined as 
\[R_a(z) := K_a(z) - z^{-1},\]
where $K_a$ is the inverse of the Cauchy transform $G$ of $a$ that satisfies $G(K_a(z)) = z$.
\end{definition}

\subsubsection{Free random variables and free convolution}

An important notion in the theory of NC probability is that of ``freeness", which in many ways is analogous to independence of standard commutative random variables. This property is described in the following definition.

\begin{definition}
Let $(\mathcal{A}, \tau)$ be an NC probability space and let $(\mathcal{A}_i)_{i=1}^m$ be subalgebras of $\mathcal{A}$ with an identity element that is the same as the identity element of $\mathcal{A}$, i.e. unital subalgebras. The $\mathcal{A}_i$ are \emph{freely independent} if for any $a_1, \ldots, a_d \in \mathcal{A}$ that satisfy
\begin{itemize}
    \item[(i)]  $\tau(a_i) = 0$ for $i \in [d]$,
    \item[(ii)] $a_i \in \mathcal{A}_{j_i}$ for $j_i \in [m]$, and
    \item[(iii)] $j_i \neq j_{i + 1}$ for $i=1, \ldots, d-1$,
\end{itemize}
it holds that $\tau(a_1 \cdots a_d) = 0$. A family of random variables $a_1, \ldots, a_d \in \mathcal{A}$ is \emph{free} if the unital subalgebras they generate are freely independent.
\end{definition}

%\VC{Should we define convergence in distribution in this section or later?} \EO{We should define it here I think. I've moved up the definition, but needs a better transition}

%We first recall that 
%the \emph{joint distribution} of a collection $a_1, \ldots, a_k \in \mathcal{A}$ is determined by the collection
%\[\tau(a_{i_1}\cdots a_{i_j}), \text{ for all } 1 \leq i_1, \ldots, i_j \leq k, j \geq 0.\]

%The analysis of free zonoids relies on understanding the distribution of linear combinations of freely independent random variables.
The distribution of the sum of two classically independent commutative random variables is the convolution of the two distributions. %, and the characteristic function satisfies
%\[\Phi_{X + Y} = \tau(e^{it(X + Y)}) = \tau(e^{itX})\tau(e^{itY}).\]
For free random variables, there is an analogous notion defined as follows. %where the distribution of the sum $a + b$ of two freely independent random variables $a$ and $b$ with distributions $\mu$ and $\nu$, respectively, is given by \emph{free convolution} of $\mu$ and $\nu$, denoted by $\mu \boxplus \nu$. % defined as follows. 
\begin{definition} 
Let $(\mathcal{A}, \tau)$ be a $C^\star$-probability space with two free self-adjoint variables $a$, $b \in \mathcal{A}$ with distributions $\mu$ and $\nu$, respectively. The \emph{free convolution} of $\mu$ and $\nu$, denoted by $\mu \boxplus \nu$, is defined to be the distribution of $a + b$ associated to $\tau$.
\end{definition}
The free convolution interacts nicely with the $R$-transform, which is additive with respect to the addition of free random variables. That is, if $a$ and $b$ are free random variables, then $R_{a + b} = R_a + R_b$.

\subsubsection{Convergence in distribution}

Next we discuss a few notions of convergence for NC random variables.

\begin{definition}\label{def:conv_in_dist}
A collection $(a_{1}^{(m)}, \ldots, a_{k}^{(m)})$ of random variables in $(\mathcal{A}_m, \tau_m)$ converges \emph{in the sense of moments} to the collection $(a_1, \ldots, a_k)$ in ($\mathcal{A}, \tau$) if for all $1 \leq i_1, \ldots, i_j \leq k$,
\begin{align*}
    \lim_{m \to \infty} \tau_m(a_{i_1}^{(m)}\cdots a_{i_j}^{(m)}) =  \tau(a_{i_1}\cdots a_{i_j}).
\end{align*}
\end{definition}

For a sequence $a^{(m)}$ of single self-adjoint, uniformly bounded NC random variables, convergence in the sense of moments is equivalent to the usual notion of convergence in distribution with respect to the spectral distributions on $\RR$. In other words, a sequence $a^{(m)} \in \mathcal{A}_m$ with uniformly bounded spectral radii\footnote{For a self-adjoint NC random variable $a \in (\mathcal{A}, \tau)$, its \emph{spectral radius} is defined by $\rho(a) := \lim_{p \to \infty} |\tau(a^{2p})|^{1/(2p)}$.} converges to a bounded, self-adjoint random variable $a \in \mathcal{A}$ in the sense of moments if and only if the sequence of spectral distributions $\nu_{m}$ of $a^{(m)}$ converges in the vague topology to the spectral distribution $\nu$ of $a$ (see, for instance, Exercise 2.5.11 in \cite{tao2023topics}).

We now highlight some existing results on the limiting distribution of random matrix ensembles discussed in Section \ref{sec:rm-examples}. All of the examples have free random variables in the limit. 

\begin{example}\label{ex:wigner_conv}
Let $A^{(m)}_1, \ldots, A^{(m)}_d$ be $m \times m$ Wigner random matrices as defined in Section \ref{sec:wigner}. %That is, they are Hermetian random matrices with diagonal and upper diagonal entries that are i.i.d. (this can be generalized to just independent with the same variance).
Then, as $m \to \infty$, $(A^{(m)}_1, \ldots, A^{(m)}_d)$ converges \emph{in distribution} to a $d$-tuple of free self-adjoint random variables $(a_1, \ldots, a_d)$ each with a semicircle distribution (Theorem 2.5 in \cite{BaiSilverstein_RandomMatrices}). %A self-adjoint NC random variable $a \in (\mathcal{A}, \tau)$ with odd moments $\tau(a^{2n+1}) = 0$ and even moments $\tau(a^{2n}) = \sigma^{2n}(2\pi)^{-1}\int_{-2\sigma}^{2\sigma} t^{2n}\sqrt{4\sigma^2 - t^2} \dint t$ has a \emph{semicircle distribution} with variance $\sigma^2$. 
The semicircle spectral distribution $\nu$ with variance $\sigma^2$ has density given by $d\nu(t) = \frac{\sigma^2}{2\pi}\sqrt{(4\sigma^2 - t^2)_+}$, which has support on the interval $[-2\sigma, 2\sigma]$. The $R$-transform of a semicircle random variable with variance $\sigma^2$ is given by $R(s) = \sigma^2 s$.  %If the entries are unbounded random variables have unbounded entries, the result may still hold (it does in the Gaussian case) using a standard truncation argument (see Section 2.1.3 in \cite{BaiSilverstein_RandomMatrices}) \EO{Check this}
\end{example}

\begin{example}\label{ex:permutation_conv}
Let $(V_1^{(m)}, \ldots, V_d^{(m)})$ be a $d$-tuple of uniformly distributed random permutation matrices as defined in Section \ref{sec:permutation}. By a result of Nica \cite{Nica1993}, $(V_1^{(m)}, \ldots, V_d^{(m)})$ converges in distribution to a $d$-tuple of free random variables $(u_1, \ldots, u_d)$ as $m \to \infty$, where the $u_i$'s are Haar unitaries. An NC random variable $u \in (\mathcal{A}, \tau)$ is a \emph{Haar unitary} if and only if $uu^* = \mathbf{1}$ and $\tau(u^k) = 0$ for all $k \in \mathbb{Z} \backslash \{0\}$.
\end{example}

\begin{example}\label{ex:wishart_conv}
Let $W_1^{(m)}, \ldots W_d^{(m)}$ be a collection of independent Wishart random matrices as defined in Section \ref{sec:wishart}. It was shown by Marchenko and Pastur \cite{marvcenko1967distribution} that for fixed $\kappa  \in (0, \infty)$, as $m \to \infty$ a single Wishart random matrix converges in distribution to the so-called \emph{Marchenko-Pastur law}. The spectral distribution of a Marchenko-Pastur random variable satisfies, for $A \subset \RR$,
\begin{align*}
\nu(A) = \begin{cases} (1 - 1/\kappa)\mathbf{1}_{\{0 \in A\}} + \mu(A), & \text{ if } \kappa > 1, \\
\mu(A), & \text{ if } 0 \leq \kappa \leq 1, \end{cases}
\end{align*}
where 
\begin{align*}
    \dint \mu(x) = \frac{\sqrt{\left(\sigma^2(1 + \sqrt{\kappa})^2 - x\right)\left(x - \sigma^2(1-\sqrt{\kappa})^2\right)}}{2\pi\sigma^2}\mathbf{1}\{x \in [\sigma^2(1-\sqrt{\kappa})^2, \sigma^2(1+\sqrt{\kappa})^2]\}.
\end{align*}
For the case of general $d$, it was proved in \cite{capitaine2004asymptotic} that
$(W_1^{(m)}, \ldots W_d^{(m)})$ converges in the sense of moments to a $d$-tuple of free random variables with the Marchenko-Pastur distribution.
\end{example}

\subsection{Noncommutative Zonoids}\label{sec:NCzonoids}

With the preceding concepts in hand, we define next the main objects of interest of this section.  To begin with, recall that random matrices are particular examples of NC random variables.  By replacing the random matrices that are used in the definition of $m$-operanoids by elements of a general $C^\star$-probability space, we extend the definition of $m$-operanoids.

\begin{definition}
    Fix a $C^\star$-probability space $(\mathcal{A}, \tau)$.  A convex body $Z$ in $\RR^d$ is a \emph{noncommutative (NC) zonoid} if it has support function of the form:
    \begin{align*}
      h_{Z}(u) := \langle v, u \rangle + \tau\left(\left|\sum_{i=1}^{d} u_i a_i\right|\right), \quad u \in \mathbb{S}^{d-1}, 
    \end{align*}
    for some $v \in \RR^d$ and some $d$-tuple $a_1, \ldots, a_d$ of self-adjoint random variables in $(\mathcal{A}, \tau)$.
If $Z$ is centered at the origin, then $v = 0$.  
Here the absolute value $|a|$ for $a \in \mathcal{A}$ is defined as $|a| := (a^\star a)^{1/2}$.

% \VC{Is something needed about origin-symmetry?  Should we perhaps have a remark after the definition?} \EO{Or we could edit the definition to match that of Definition 31 for lift NC zonoids as we did with operanoids so we cover the general center setting?} \VC{Let's discuss this!}
%If $(A_1, \ldots, A_d)$ are \emph{free}, then $Z_A$ is called a \emph{free zonoid}. 
\end{definition}

% \EO{In the definition of a NC zonoid, we assume the random variables are self-adjoint to have access to their spectral distribution of the spectral distribution of their linear combination, as well as the usual notion of convergence in distribution described in the previous section.} \VC{This is fine now given that earlier we use Hermitian matrices, right?}

In the following examples, we illustrate that operanoids can be recovered as examples of NC zonoids for a particular choice of the underlying $C^\star$-probability space.

\begin{example}\label{ex:rand_matrix_Calg}
Let $(\mathcal{A}, \tau)$ be the space of random Hermitian matrices with bounded entries $\mathbb{H}^m \otimes L^{\infty}(\Omega, \mathcal{F}, \PP)$ equipped with the functional $\tau(A) = \EE[ \mathrm{tr}_m(A)]$. This is a $C^\star$-probability space with involution given by the transpose of the random matrix $(A(\cdot))^{\star} = A^\star(\cdot)$. For a self-adjoint $m \times m$ random matrix $A$ with bounded entries, the spectral distribution $\nu$ of $A$ as characterized in \eqref{eq:spectalmeasure} is given by $\nu = \EE\left[m^{-1}\sum_{j=1}^m \delta_{\lambda_j(A)}\right]$, where $\lambda_1(A) \leq \cdots \leq \lambda_m(A)$ are the eigenvalues of $A$ (see Example 1.6.2 in \cite{Voiculescu_lectures}).
NC zonoids with respect to this $C^\star$-probability space are $m$-operanoids.
\end{example}

As we articulated previously, a primary motivation for defining NC zonoids is that the framework of NC probability allows us to make sense of the notion of convergence of $m$-operanoids as $m$ grows.  In particular, this question is closely related to that of the limiting distribution of a $d$-tuple of $m \times m$ random matrices as $m$ grows.  As the next result shows, under certain conditions the limiting object will be an NC zonoid.

\begin{proposition}\label{prop:distr_limit}
Consider a sequence of $C^\star$-probability spaces $(\mathcal{A}_m, \tau_m)$ for $m \in \mathbb{N}$.  For each $m$, let $a^{(m)}_1, \ldots, a^{(m)}_d$ be self-adjoint NC random variables in $(\mathcal{A}_m, \tau_m)$, and assume the collection $\{a_i^{(m)}\}_{i\in [d], m \in \mathbb{N}}$ is uniformly bounded.  Suppose that $(a^{(m)}_1, \ldots, a^{(m)}_d)$ converges in the sense of moments to a $d$-tuple of NC random variables $(a_1, \ldots, a_d)$, each lying in a $C^\star$-probability space $(\mathcal{A}, \tau)$, as $m \to \infty$. Let $Z_m$ be the NC zonoid associated to $(a^{(m)}_1, \ldots, a^{(m)}_d)$ and $Z$ the NC zonoid associated to $(a_1, \ldots, a_d)$. Then,
\[ \lim_{m \to \infty} d_H(Z_m, Z) = 0.\]
\end{proposition}

\begin{proof}
The convergence in the sense of moments of $(a^{(m)}_1, \ldots, a^{(m)}_d)$ to $(a_1, \ldots, a_d)$ as $m \to \infty$ implies that for all $u \in \mathbb{S}^{d-1}$, $\sum_{i=1}^d u_i a_i^{(m)}$ converges in the sense of moments to $\sum_{i=1}^d u_i a_i$ as $m \to \infty$. Under the assumptions, this is equivalent to convergence in distribution with respect to the spectral distributions (see discussion after Definition~\ref{def:conv_in_dist}).  Thus, 
we have for each $u \in \mathbb{S}^{d-1}$ that:
\begin{align*}
\lim_{m \to \infty} h_{Z_m}(u) = \lim_{m \to \infty} \tau_m\left(\left|\sum_{i=1}^d u_i a^{(m)}_i\right|\right) 
&= \tau\left(\left|\sum_{i=1}^d u_i a_i\right|\right) = h_Z(u). 
\end{align*}
The result then follows from Theorem 1.8.15 in \cite{schneider_2013}.
\end{proof}

The above proposition says that under suitable assumptions, convergence in distribution of a tuple of NC random variables implies convergence of the associated NC zonoids in the Hausdorff metric. While Proposition \ref{prop:distr_limit} provides a sufficient condition for a sequence of $m$-operanoids to converge to a NC zonoid, we should not expect the converse of the result to be true in general.  That is, even if a sequence of NC zonoids converges in Hausdorff distance, the associated sequence of tuples of NC random variables may not exhibit convergence in distribution.  This is because an NC zonoid does not uniquely characterize the distribution of the underlying tuple of NC random variables; in fact, even a zonoid is not in one-to-one correspondence with the distribution of the underlying tuple of scalar random variables.

We next illustrate examples of NC zonoids that can be obtained as limits of $m$-operanoids using Proposition \ref{prop:distr_limit}.
As a first example, we show that NC zonoids obtained from $d$-tuples of free semicircle random variables are ellipsoids with primary axes along the standard bases. Recall that the $R$-transform is additive with respect to the addition of free random variables. Also, observe that for $t \in \RR$, $R_{ta}(s) = tR_a(ts)$. Thus, for free NC random variables $(a_1, \ldots, a_d)$ and a unit vector $u \in \mathbb{S}^{d-1}$, the linear combination $\sum_{i=1}^d u_i a_i$ has $R$-transform
\begin{align}\label{eq:R_lincombo}
    R_{\sum_{i=1}^d u_ia_i}(s) := \sum_{i=1}^d u_i R_{a_i}(u_i s).
\end{align}

\begin{example}\label{ex:wigner_ellipsoid}
Recall from Example \ref{ex:wigner_conv} that a $d$-tuple of independent Wigner random matrices converges in distribution to a $d$-tuple of free semicircle random variables. To see what the associated limiting NC zonoid is, let $u \in \mathbb{S}^{d-1}$ and $(a_1, \ldots, a_d)$ be a $d$-tuple of free random variables such that $a_i$ has a semicircle distribution with variance $\sigma_i$ as defined in Example \ref{ex:wigner_conv}.  We first observe that $\sum_{i=1}^d u_ia_i$ has a semicircle distribution with variance $\sqrt{\sum_{i=1}^d u_i^2\sigma_i^2}$. Indeed, for a semicircle random variable $a$ and for $u \in \RR$, $ua$ has a semicircle distribution with variance $u^2$.  Using \eqref{eq:R_lincombo}, the $R$-transform of $\sum_{i=1}^d u_i a_i$ is then given by
    \begin{align*}
        R_{\sum_{i=1}^d u_ia_i}(s) =  \sum_{i=1}^d u_i^2\sigma_i^2 s.
    \end{align*}
Letting $\sigma(u) := \left(\sum_{i=1}^d u_i^2\sigma_i^2\right)^{1/2}$, we obtain that
\begin{align*}
    \tau\left(\left|\sum_{i=1}^d u_ia_i\right|\right) &= (2\pi \sigma(u)^2)^{-1} \int_{-2 \sigma(u)}^{2\sigma(u)} |t| \sqrt{4\sigma(u)^2 - t^2} \dint t 
    %&=  (\sigma(u)^2\pi)^{-1} \int_{0}^{2\sigma(u)} t \sqrt{4\sigma(u)^2 - t^2} \dint t \\ 
    %&=  2(\sigma(u)\pi)^{-1} \int_{0}^{2\sigma(u)} t \sqrt{1 - (t/2\sigma(u))^2} \dint t \\ 
    =  8\sigma(u)\pi^{-1} \int_{0}^{1} r \sqrt{1 - r^2} \dint r = c\sigma(u),
\end{align*}
where $c > 0$ is a constant. Thus, by Proposition \ref{prop:distr_limit}, $m$-operatopes associated to $d$-tuples of Wigner random matrices converge in the Hausdorff metric to NC zonoids that are axis-aligned ellipsoids as $m \rightarrow \infty$.

\end{example}

\begin{example}\label{ex:perm_haarNCzonoid}

Recall from Example \ref{ex:permutation_conv} that a $d$-tuple of independent uniformly distributed random permutation matrices converge in distribution to a $d$-tuple of free Haar unitaries.
While we do not have a closed form expression for the support function of a NC zonoid corresponding to a $d$-tuple of free Haar unitaries, we plot in Figure \ref{fig:haar-nczonoid} an $m$-operatope approximation associated to a $d$-tuple of independent uniform permutation matrices for large $m$.
\end{example}

\begin{example}\label{ex:wish_mpNCzonoid}

Recall from Example \ref{ex:wigner_conv} that a $d$-tuple of independent Wishart random matrices converge in distribution to a $d$-tuple of free random variables with a Marchenko-Pastur distribution.
Again, we do not have a closed form expression for the support function of the corresponding NC zonoid, and we plot in Figure \ref{fig:haar-nczonoid} an $m$-operatope approximation associated to a $d$-tuple of Wishart random matrices for large $m$.
\end{example}

%\EO{Add plot of m-operanoids for uniform permutation matrices for large m.}

\begin{figure}[h]
    \centering
    \includegraphics[width=0.25\linewidth]{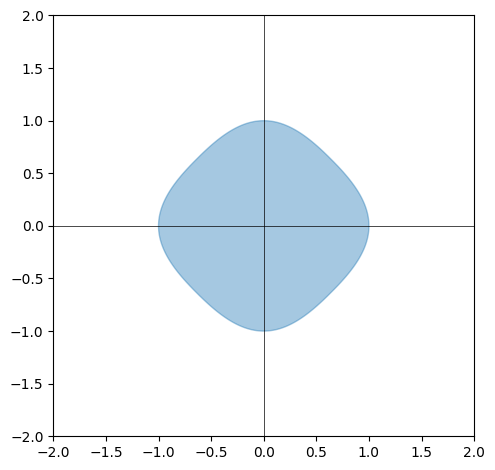}\hspace{0.5in}
    \includegraphics[width=0.25\linewidth]{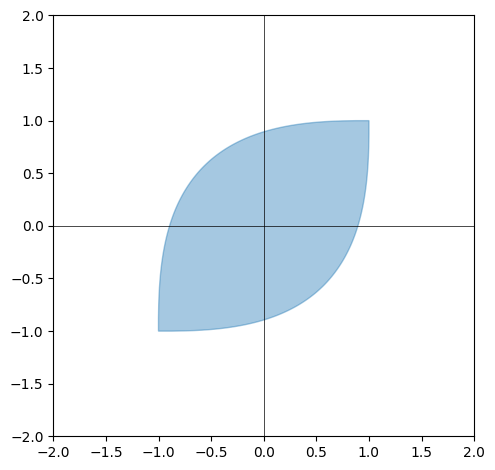}
    \caption{Approximation of the limiting NC zonoids in $\RR^2$ associated to free Haar unitaries (left) and to free Marchenko-Pastur random variables (right).  These are each a random sample of an $m$-operatope with $N = 1$ summand and with $m=500$ associated to independent uniform permutation matrices and to independent Wishart random matrices, respectively.}
    \label{fig:haar-nczonoid}
\end{figure}

It is of interest to study NC zonoids in general dimensions $d$, which is a convex body parameter of the \emph{joint} distribution of a $d$-tuple of self-adjoint NC random variables.  The distribution of a single self-adjoint NC random variable can be defined by its spectral distribution, which is given by a measure on $\RR$.  For $d = 2$, we are not aware of a notion of \emph{joint spectral measure} of a pair of NC random variables, although this notion is available in the case of a pair of matrices~\cite{heinavaara2025tracial}.  Based on the observation that all NC zonoids are zonoids in $\RR^2$, there is a measure $\nu_{a_1, a_2}$ on $\RR^2$ such that for all $(u_1, u_2) \in \mathbb{S}^{1}$, $$\tau(|u_1a_1 + u_2 a_2|) = \int_{\RR^2} |u_1x_1 + u_2x_2| \dint \nu_{a_1, a_2}(x).$$ For such a measure $\nu_{a_1,a_2}$ to be viewed as a joint spectral measure associated to $(a_1,a_2)$, it should be the case that:
    \begin{align*}
        \tau(f(a_1, a_2)) = \int_{\RR^2} f(x) \dint \nu_{a_1,a_2}(x),
    \end{align*}
    for some more general class of functions $f$ beyond just the absolute value $|\cdot|$.  It would be of interest to define and study such joint spectral measures of pairs of NC random variables.

\subsection{Free Zonoids}\label{sec:freezonoids}

Examples \ref{ex:wigner_ellipsoid}--\ref{ex:wish_mpNCzonoid} yield NC zonoids that are specified by tuples of free random variables.  Motivated by this observation, we incorporate the assumption of freeness into the definition of NC zonoids to obtain the following subclass of NC zonoids called \emph{free zonoids}. Due to the additional structure imposed via freeness, we are able to obtain explicit characterizations of the support functions of these convex bodies in a few cases, as in Example \ref{ex:wigner_ellipsoid} above and Example \ref{example:non-ellipsoidal} in the sequel.

\begin{definition}
Let $(a_1, \ldots, a_d)$ be a $d$-tuple of free random variables in a $C^\star$-probability space $(\mathcal{A}, \tau)$ and fix some $v \in \RR^d$.  The associated NC zonoid with support function
\[h_{Z}(u) := \langle v, u \rangle + \tau\left( \left|\sum_{i=1}^d u_i a_i\right| \right), \quad u \in \mathbb{S}^{d-1},\]
is called a \emph{free zonoid}.
% \VC{Again, do we need to make a note of origin symmetry?} \EO{Same comment as for definition 17} \VC{Let's discuss this!}
\end{definition}

An important point of distinction between free zonoids and NC zonoids (as well as operatopes and operanoids) is that the class of free zonoids is not closed under affine transformations.  The reason for this is that a $d$-tuple of free random variables does not remain free under the action of a linear map.

As an additional example beyond Examples \ref{ex:wigner_ellipsoid}--\ref{ex:wish_mpNCzonoid}, we next consider a free zonoid obtained by taking linear combinations of free, but not identically distributed, random variables.
% \VC{For the following example, the notion of a `spectral measure' comes up.  Would it be useful to define this more generally?} \EO{Can you remind me your thoughts here? Is the description in Section 4.1.1 sufficient?} \VC{Let's discuss -- I wonder what the right place is to bring up the point about joint spectral measures for NC random variables.}

\begin{example}\label{example:non-ellipsoidal}
Let $u \in \mathbb{S}^1$ and $\beta \in \RR$. Consider the random matrix
\[B = u_1 B_1 + u_2 \beta B_2,\]
where $B_1$ is a shifted $\kappa m \times \kappa m$ Wishart matrix satisfying
\[B_1 = \frac{1}{m \sqrt{\kappa}} W^TW -   I,\]
for an i.i.d. Gaussian matrix $W \in \RR^{m \times \kappa m}$ and $\kappa m \times \kappa m$ identity matrix $I$, and $B_2$ is an $\kappa m \times \kappa m$ GOE matrix. Asymptotically as $m \to \infty$, $B$ converges to the free additive convolution of a (shifted) Marchenko-Pastur law and a semicircle law.

The support function of the associated free zonoid is given by
$h_Z(u) = \int_{\RR} |x| \rho(x;u) \dint x$, where the explicit form of the density $\rho(x;u)$ may be obtained by appealing to \cite[(A.2)]{cheng2013spectrum} through the connection observed in \cite{LuYao2025_equivpolyspec}.  Specifically, in terms of the reparametrization $\gamma = \frac{1}{\kappa}$, $a = \frac{u_1}{\kappa^{1/2}}$, and $\nu = \frac{1}{\kappa}\left( u_1^2 + (1 - u_1^2) \beta^2\right)$, %the explicit form of $\rho(x;u)$ is 
\begin{align*}
    \rho(x;u) = \begin{cases} 0, \, D \leq 0 \\ \sqrt{\frac{3}{2}}((\sqrt{D} + R)^{1/3} + (\sqrt{D} - R)^{1/3}), \, D > 0, \end{cases} 
\end{align*}
where
\begin{align*}
    D &= Q^3 + R^2, &
    R &= (9\alpha_2\alpha_1 - 27 \alpha_0 - 2\alpha_2^3)/54, &
    Q &= (3\alpha_1 - \alpha_2^2)/9,
\end{align*}
and 
\begin{align*}
    \alpha_2 &= \frac{(\nu + ax)\gamma}{a(\nu - a^2)}, &
    \alpha_1 &= \frac{(a + \gamma x)\gamma}{a(\nu - a^2)}, &
    \alpha_0 &= \frac{\gamma^2}{a(\nu - a^2)}.
\end{align*}
Figure~\ref{fig:placeholder} gives illustrations of these free zonoids for several choices of $\kappa$ and $\beta$.
\end{example}

\begin{figure}[h]
    \centering
    \includegraphics[width=0.25\linewidth]{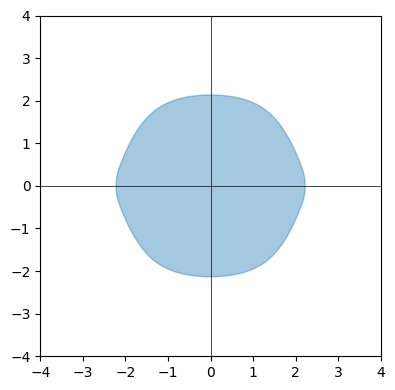}
    \hfill\includegraphics[width=0.25\linewidth]{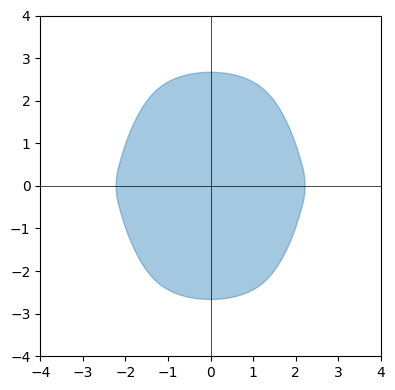}
    \hfill\includegraphics[width=0.25\linewidth]{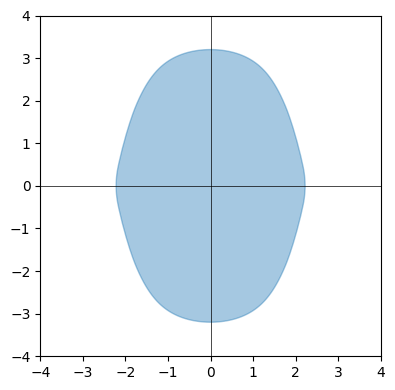}
    \caption{NC zonoid in $\R^2$ from Example~\ref{example:non-ellipsoidal} for parameters $\kappa = 5$, $\beta = .75$, $1$, and $1.25$ from left to right.}
    \label{fig:placeholder}
\end{figure}

%\VC{Should the beginning part here go to the beginning of the subsection on free zonoids as motivation?}

In the preceding example as well as in Examples \ref{ex:wigner_ellipsoid}--\ref{ex:wish_mpNCzonoid}, a free zonoid is obtained as a limit of a sequence of $m$-operanoids.  The next result states that such a construction is in fact applicable for any free zonoid, as all such zonoids may be obtained as limits of sequences of $m$-operanoids for growing $m$.

% A natural problem is to characterize the convex bodies, or more specifically the NC zonoids, that can be approximated as the limit of $m$-operanoids as $m \to \infty$. 
% This remains an open challenge in general, but we do have the following result showing that any free zonoid can be obtained as the limit of operanoids. 

\begin{proposition}\label{prop:freelimit}
    Let $Z$ be a free zonoid. There exists a sequence of $m$-operanoids $\{Z_m\}_{m =1}^{\infty}$ such that
    \begin{align*}
        \lim_{m \to \infty} d_H(Z_m, Z)  = 0.
    \end{align*}
\end{proposition}

Before giving the proof, we recall that Haar unitaries (see Example \ref{ex:permutation_conv} for definition) have the following \emph{liberation} property: Let $(\mathcal{A}, \tau)$ be a $C^{\star}$-probability space and $a,b,u \in \mathcal{A}$ such that $u$ is a Haar unitary and $a,b$ are self-adjoint random variables that are free from $u$. Then, $ubu^\star$ has the same distribution as $b$ and is free from $a$.
We can now prove Proposition \ref{prop:freelimit}.

\begin{proof}
Let $Z$ be a free zonoid parameterized by a $d$-tuple $(a_1, \ldots, a_d)$ of free random variables, where the spectral distribution of $a_i$ is denoted by $\nu_i$. For each $m$, let $D_1^{(m)}, \ldots, D_d^{(m)}$ be a $d$-tuple of independent random diagonal matrices, where the diagonal entries of $D_i^{(m)}$ are i.i.d. random variables with distribution $\nu_i$ on $\RR$ normalized by $\frac{1}{m}$. We can apply the asymptotic liberation technique to obtain a $d$-tuple of free random matrices with the same marginal distributions:
$$\left(U_1^{(m)}D^{(m)}_1(U_1^{(m)})^{\star}, \ldots, U_d^{(m)}D^{(m)}_d(U_d^{(m)})^{\star}\right),$$
where the $U_i^{(m)}$ are Haar unitary random matrices that are free from each other and from the diagonal matrices $D_i^{(m)}$.
% \VC{Instead of saying ``... free and free ...'', how about saying ``... free from each other and from ...''?} \EO{yes, that sounds better, done.}
Thus, as $m \to \infty$, this $d$-tuple converges in distribution to
$(a_1, \ldots, a_d)$ of free random variables in $(\mathcal{A}, \tau)$ such that $a_i$ has spectral distribution $\mu_i$. For each $m$, let $Z_m$ be the free zonoid parameterized by the $d$-tuple above. The final conclusion is then obtained by applying  Proposition~\ref{prop:distr_limit}.
% \VC{Should we add an extra sentence here applying Proposition~\ref{prop:distr_limit} to conclude the result?} \EO{Yes, I added a few sentences to clarify.}
\end{proof}

\begin{remark}
    The question of which NC zonoids can be approximated by operanoids is related to the Connes embedding conjecture, which was recently proved to be false \cite{Connes2021}. In particular, a result of Voiculescu \cite{Voiculescu1994} says that if the Connes embedding conjecture is false, then there exists a $d$-tuple of NC random variables $(a_1, \ldots, a_d) \in (\mathcal{A}, \tau)$ that cannot be approximated in distribution by $d$-tuples of $m \times m$ matrices as $m \to \infty$. This does not necessarily imply the existence of NC zonoids that cannot be approximated by limits of $m$-operanoids, since the convergence of NC zonoids in Hausdorff distance does not imply convergence in distribution of the underlying random variables, so this remains an interesting open problem. 
\end{remark}

\section{Towards Applications}\label{sec:applications}

In this section we take preliminary steps towards applications by discussing operanoid and NC zonoid generalizations of previous methods that use zonoids.  Specifically, in Section~\ref{sec:lift-NC} we present an analog of a lift zonoid for representing high-dimensional data distributions and in Section~\ref{sec:poisson-hypersurface} we describe applications of operanoids in stochastic geometry.  Along the way, we discuss some open questions for further investigation.

\subsection{Representations of NC Probability Distributions} \label{sec:lift-NC}

From their definitions, an operanoid or an NC zonoid may be viewed as a parameter of the distribution of the underlying $d$-tuple of random matrices or NC random variables. Clearly, multiple distributions correspond to the same operanoid or NC zonoid. However, in the case of zonoids, there is a way to define a \emph{unique} zonoid for a given random vector in $\RR^d$ that can be used as a representation of its probability distribution. 
In this section, we first describe this representation, called the lift zonoid, which comes with a corresponding metric and stochastic ordering on the space of probability measures on $\RR^d$. We will then define a new analytic object specified by the joint distribution of a $d$-tuple of random matrices, and more generally NC random variables, via the operanoid and NC zonoid generalization of the lift zonoid. %This new representation is a generalization of an object called the \emph{lift zonoid}, which is a representation of random vectors in $\RR^d$ with many applications in statistics.

\subsubsection{Lift zonoids}
Given a probability measure on $\RR^d$, there exist many representations of the measure, such as the cumulative distribution function (CDF) or the characteristic function, and each of these has appealing properties that are useful in various contexts.  For instance, the CDF is used to define the Kolmogorov-Smirnov metric between probability distributions on $\RR^d$, as well as a stochastic ordering based on an order of these functionals.  Here we recall a convex-geometric representation of probability measures based on zonoids that exhibits many appealing properties.  Specifically, for a probability measure $\mu$ on $\RR^d$, the \emph{lift zonoid} is the convex body in $\RR^{d+1}$ obtained by taking the expectation of the line segment $\EE\left([0, (1, X)]\right)$, where $X \sim \mu$ \cite{KoshevoyMosler_Lift_1998}. That is, the lift zonoid $Z_{\mu}$ has support function
\begin{align*}
h_{Z_{\mu}}(u) := \int_{\RR^d} \max\{0, \langle u, (1,x) \rangle\} \dint \mu(x).
\end{align*}
Note that this zonoid is not centered at the origin, but instead has center of symmetry at $\frac{1}{2}(1, \int_{\mathbb{R}^d} x \dint \mu(x))$.
There is a one-to-one correspondence between a multivariate probability distribution and its lift zonoid, so this convex body is a representation of the underlying probability measure. 

The lift zonoid can be used to define a metric between probability measures and a stochastic ordering. Indeed, a metric between probability measures is defined by the Hausdorff distance between the two corresponding lift zonoids, i.e.
\begin{align*}
    d(\mu, \nu) = d_H(Z_{\mu}, Z_{\nu}) := \sup_{u \in \mathbb{S}^d} |h_{Z_{\mu}}(u) - h_{Z_{\nu}(u)}|.
\end{align*}
A stochastic ordering is defined through containment of the lift zonoids, i.e. $\mu \leq \nu$ if and only if $Z_{\mu} \subseteq Z_{\nu}$. The lift zonoid representation also leads to a notion of central regions and data depth \cite{mosler_multivariate_2002}. 

An advantage of the lift zonoid representation over the CDF is that it behaves nicely under certain transformations of the probability measure. For example, the lift zonoid is equivariant with respect to linear transformations of the probability measure (see Proposition 2.24 in \cite{mosler_multivariate_2002}).  %Moreover, the lift zonoid of a sum of probability measures is the Minkowski sum of their lift zonoids. 
% \VC{In the preceding sentence, should we replace the two occurrences of ``sum'' by ``linear combination''?  Alternatively, we could just change ``probability measure'' by ``measure''?  I ask because ``sum of probability measures'' feels strange.} \EO{I actually think we don't need that sentence (I commented it out), it felt repetative with respect to equivariant sentence before it - what do you think?}
For more properties of the lift zonoid and its applications in statistics, we refer the reader to the monograph \cite{mosler_multivariate_2002}.

\subsubsection{Lift operanoids and lift NC zonoids}

There is a natural generalization of a lift zonoid for probability measures on the space of $d$-tuples of self-adjoint matrices in $\mathbb{H}^m$. First, we note that we can define an operanoid (with center not necessarily at the origin) as a convex body $Z$ with support function
\begin{align}\label{e:hK_offcenter}
    h_Z(u) = \EE\left[\mathrm{Tr}\left(\left(\sum_{i=1}^d u_i A_i\right)_+\right)\right], \quad u \in \mathbb{S}^{d-1},
\end{align}
where $(A)_+$ is the positive part of the matrix $A$.  If $Z$ has a center of symmetry $x_0$, we can center $Z$ at the origin by defining $Z_0 := Z - x_0$. As in Corollary 2.4 in \cite{bolker_class_1969}, we see that the support function of $Z_0$ is:
\begin{align*}
    h_{Z_0}(u) &= \frac{1}{2}\left(h_{Z_0}(u) + h_{Z_0}(-u)\right) = \frac{1}{2}\left(h_Z(u) - \langle x_0, u\rangle + h_{Z}(-u) + \langle x_0, u \rangle \right)\\
    &= \frac{1}{2}\EE\left[\mathrm{Tr}\left(\left(\sum_{i=1}^d u_i A_i\right)_+\right)\right] + \frac{1}{2}\EE\left[\mathrm{Tr}\left(\left(\sum_{i=1}^d - u_i A_i\right)_+\right)\right] %\\
    %&= \frac{1}{2}\EE\left[\mathrm{Tr}\left(\left(\sum_{i=1}^d u_i A_i\right)_+\right)\right] - \frac{1}{2}\EE\left[\mathrm{Tr}\left(\left(\sum_{i=1}^d u_i A_i\right)_-\right)\right] \\
    = \frac{1}{2}\EE\left[\left\|\sum_{i=1}^d u_i A_i \right\|_1\right],
\end{align*}
which corresponds with our original definition of operanoids centered at the origin. A convex body $Z$ defined above is indeed centrally symmetric about some $x_0$, as proved in the following proposition.

\begin{proposition}\label{prop:center}
The operanoid $Z$ with support function \eqref{e:hK_offcenter} is centrally symmetric around $\frac{1}{2}\sum_{j=1}^m \EE[\mathbf{a}^{(j)}]$, where $\mathbf{a}^{(j)} = (A_1^{(jj)}, \ldots, A_d^{(jj)}) \in \RR^d$ is the vector of the $j$th diagonal entries of the parametrizing matrices.
% \VC{It is unclear why $K$ should contain the origin if it is not centered.  I'm a bit confused by this.} \EO{I deleted that my mistake - the lift zonoid/operanoid contains the origin, but this proposition just concerns a general shifted operanoid, so it doesn't have to contain the origin.}
\end{proposition}

\begin{proof}
    For $t \in \RR$, we have $\max\{t,0\} - \frac{1}{2}|t| = \frac{1}{2}t$, which implies that:
    \begin{align*}
    \mathrm{Tr}\left(\left(\sum_{i=1}^d u_i A_i\right)_+\right) &= \sum_{j=1}^m \max\left\{\lambda_j\left(\sum_{i=1}^d u_i A_i \right), 0 \right\} = \sum_{j=1}^m \frac{1}{2}\lambda_j\left(\sum_{i=1}^d u_i A_i \right) + \frac{1}{2}\sum_{j=1}^m \left|\lambda_j\left(\sum_{i=1}^d u_i A_i \right) \right| \\
        &= \frac{1}{2} \mathrm{Tr}\left(\sum_{i=1}^d u_i A_i\right) + \frac{1}{2}\left\|\sum_{i=1}^d u_i A_i \right\|_1 = \frac{1}{2} \sum_{j=1}^m \langle u,  \mathbf{a}^{(j)}\rangle + \frac{1}{2}\left\|\sum_{i=1}^d u_i A_i\right\|_1.
    \end{align*}
    Taking expectations of both sides gives $Z = Z_0 + \frac{1}{2}\sum_{j=1}^m \EE[\mathbf{a}^{(j)}]$ for $Z_0$ that is centrally symmetric about the origin, which proves the claim.
\end{proof}
% \VC{In the preceding proposition statement and proof, should we change $K, K_0$ to $Z_, Z_0$ to match the preceding and successing discussion?} \EO{Yes, done.}

We now define a convex body associated to a collection of self-adjoint random matrices that generalizes the notion of a lift zonoid.

\begin{definition}
    Let $A_1, \ldots, A_d$ be a $d$-tuple of random matrices in $\mathbb{H}^m$. The %(centered) 
    associated $\emph{lift operanoid}$ is the convex body in $\RR^{d+1}$ with support function
    \begin{align*}
      %h_{Z_A}(u) := \frac{1}{2}\mathbb{E}\left[\left\|u_1I + \sum_{i=2}^{d+1} u_i A_i\right\|_1\right],  
      h_{Z}(u) := \mathbb{E}\left[\mathrm{Tr}\left(u_1I + \sum_{i=2}^{d+1} u_i A_i\right)_+\right]
    \end{align*}
    where $I$ is the $m \times m$ identity matrix.
\end{definition}

More generally, for a $d$-tuple of NC random variables, we define the following generalization of a lift zonoid representation.

\begin{definition}
    Let $a_1, \ldots, a_d$ be a $d$-tuple of NC random variables in a $C^{\star}$ probability space $(\mathcal{A}, \tau)$. The associated $\emph{lift NC zonoid}$ is the convex body in $\RR^{d+1}$ with support function
    \begin{align*}
      h_{Z}(u) := \frac{1}{2}\tau\left(\left|u_1\mathbf{1} + \sum_{i=2}^{d+1} u_i a_i\right|\right) + \frac{1}{2}\tau\left(u_1\mathbf{1} + \sum_{i=2}^{d+1} u_i a_i\right),  
    \end{align*}
    where $\mathbf{1}$ is the identity element of $\mathcal{A}$.
\end{definition}

As described previously, the utility of the lift zonoid representation stems from the fact that this convex body uniquely characterizes the distribution of the underlying $d$-tuple of random variables. Further, one can use this representation to define a metric on the space of probability on $\RR^d$ which has a variety of applications. Therefore, we would first like to answer the following basic questions.

\begin{question} Does the lift NC zonoid uniquely characterize the joint distribution of the underlying NC random variables?
\end{question}

This question remains open in general, but there are a few cases for which we can answer this question positively because the lift NC zonoid corresponds to a lift zonoid.
First, consider the case $d=1$ and a single deterministic self-adjoint NC random variable $a \in (\mathcal{A}, \tau)$. Its associated lift NC zonoid $Z$ in $\RR^2$ has support function
\begin{align*}
h_Z(u) &= \frac{1}{2}\tau\left(|u_1\mathbf{1} +  u_2 a|\right) + \tau(u_1\mathbf{1} + u_2 a) = \frac{1}{2}\int_{\RR} |u_1 + u_2 x|\dint \nu(x) + \frac{1}{2}\int_{\RR} (u_1 + u_2 x )\dint \nu(x) \\
&= \int_{\RR} \max\{0, u_1 + u_2 x\}\dint \nu(x),
%h_Z(u) &= \frac{1}{m}\mathrm{Tr}\left(u_1I +  u_2 A\right)_+ \\
%&= \frac{1}{m}\sum_{j=1}^m \max\{0, u_1 + u_2 \lambda_j(A)\} \\
%&= \int_{\RR} \max\{0, u_1 + u_2 x\}\dint \mu_A(x),
\end{align*}
where %$\mu_A := \frac{1}{m} \sum_{j=1}^m \delta_{\lambda_j(A)}$ 
$\nu$ is the spectral distribution of $a$ as defined in Section \ref{sec:NCrandomvariables}.
Thus, $Z$ is the lift NC zonoid associated to $\nu$. By Theorem 2.21 in \cite{mosler_multivariate_2002}, $Z$ uniquely characterizes the distribution $\nu$. %Since $\mu_A$ is a discrete measure, the associated lift NC zonoid will be a zonotope.

%Similarly, consider a self-adjoint random $m \times m$ matrix $A$. The associated lift operanoid is
%\begin{align*}
%h(Z(A), u) &= \EE\left[\frac{1}{m}\mathrm{Tr}\left(u_1I +  u_2 A\right)_+\right] \\
%&= \EE\left[ \frac{1}{m}\sum_{j=1}^m \max\{0, u_1 + u_2 \lambda_i(A)\} \right] \\
%&= \EE\left[\int_{\RR} \max\{0, u_1 + u_2 x\}\dint \mu_A(x) \right] = \int_{\RR} \max\{0, u_1 + u_2 x\}\dint \Lambda_A(x).
%\end{align*}
%Thus, $Z(A)$ is the lift zonoid associated to the distribution $\Lambda_A$ of $A$ as described in Example \ref{ex:rand_matrix_Calg}. Again by Theorem 2.21 in \cite{mosler_multivariate_2002}, $Z(A)$ uniquely characterizes the distribution $\Lambda_A$. Also note that $Z(A)$ symmetric around $\frac{1}{2}(1, \int x \dint\nu(x))$ by Proposition \ref{prop:center}, where $\nu$ is the distribution of $A$.

%More generally, if $a = a^* \in \mathcal{A}$ is a \emph{self-adjoint} noncommutative random variable, then its distribution $\mu_a$ can be described by a compactly supported probability measure on $\RR$. The lift NC zonoid then corresponds to the lift zonoid of $\mu_a$ (see \cite[p.287]{mosler_multivariate_2002}), and thus uniquely characterizes the distribution of $a$.
% $\alpha(\mu) = \int_{\RR^d} \mu(dx)$ and $\ee(\mu) = \int_{\RR^d} x \mu(dx)$

%Question: Is the lift NC zonoid of a pair of matrices $(A,B)$ the lift zonoid of the joint spectral measure $\mu_{A,B}$ as in \cite{Heinavaara2023}? If yes, then these lift NC zonoids are all zonoids.

Now consider the case $d=2$ and two matrices $A$ and $B$ in $\mathbb{H}^m$. The associated lift NC zonoid $Z$ in $\RR^2$ has support function
\begin{align*}
h_Z(u) &= \frac{1}{m}\mathrm{Tr}\left(u_1I +  u_2 A + u_3 B\right)_+ = \int_{\RR^2} \max\{0, u_1 + u_2 x + u_3 y\} \dint \mu_{A,B}(x,y),
\end{align*}
where $\mu_{A,B}$ is the tracial joint spectral measure associated to $A$ and $B$ as in Theorem 1.4 in \cite{heinavaara2025tracial}. This shows that $Z$ is the lift zonoid in $\RR^3$ of the measure $\mu_{A,B}$ on $\RR^2$. Thus, 
the tracial joint spectral distribution of $A$ and $B$ is uniquely determined by the lift NC zonoid.

\begin{question} What is the relationship between convergence in the Hausdorff metric of the lift NC zonoid to convergence in the sense of moments for a collection of NC random variables?
\end{question}

In Section~\ref{sec:NCzonoids}, we discussed how convergence in the sense of moments implies convergence of the associated NC zonoids in the Hausdorff metric for a collection of NC random variables, but clearly the converse implication is not true. However, if the lift NC zonoid represents
% \VC{Should we change ``characterizes'' to something else such as ``represents''?  We want to signify that the lift NC zonoid is in one-to-one correspondence with the distribution, right?} \EO{Yes, we do and I meant characterizes that way but I like "represents" too}
the distribution of a $d$-tuple of self-adjoint NC random variables, we can study the metric measuring the distance between NC random variables using the Hausdorff distance between their associated lift NC zonoids. In particular, it would be of interest to generalize the well-known result for lift zonoids that convergence of the associated lift zonoids implies convergence in distribution of the associated random vectors (under mild conditions) \cite[Theorem 2.30]{mosler_multivariate_2002}.

%\EO{How much do we want to say here?}
%\EO{ADD NOTE ON METRIC. Convergence w.r.t this metric is implied by convergence in distribution if the RVs are uniformly integrable in the znoid case. And if the RVs have a joint density, then convergence in distribution implies convergence in the lift zonoid metric.}

\subsection{Stationary Poisson Hypersurfaces} \label{sec:poisson-hypersurface}

In this section, we explore potential applications of operanoids in stochastic geometry, where zonoids play an important role in parameterizing random geometric models. We first review the specific application of zonoids to random hyperplane models and then describe a new model for random hypersurfaces with an analogous connection to operanoids.

\subsubsection{Poisson hyperplane process}

A \emph{Poisson hyperplane process} is a Poisson point process in the space $\mathcal{H}^d$ of hyperplanes in $\RR^d$. We parameterize a hyperplane $H \in \mathcal{H}^d$ by its normal vector $u \in \mathbb{S}^{d-1}$ and distance $t \in \RR$ and define $H := H(u,t) = \{x \in \RR^d : \langle x, u \rangle = t\}$. A random hyperplane process can then be constructed as follows. Let $\Phi = \{t_j\}_{j \in \mathbb{Z}}$ be a homogeneous Poisson process in $\RR$ with parameter $\gamma > 0$ and $\{u_j\}_{j \in \mathbb{Z}}$ be a collection of i.i.d. unit vectors independent of $\Phi$ with distribution $\phi$, an even probability measure on the unit sphere.  The collection of hyperplanes $X = \{H(u_j,t_j)\}_{j \in \mathbb{Z}}$ defines a \emph{stationary Poisson hyperplane process} with intensity $\gamma$ and directional distribution $\phi$. We refer to the monograph \cite{hug2024poisson} for further background on these models.

The \emph{associated zonoid} of the stationary Poisson hyperplane process $X$ is the zonoid $\Pi$ with support function
\begin{align*}
    h_{\Pi}(v) := \frac{\gamma}{2}\int_{\mathbb{S}^{d-1}}|\langle u, v \rangle| \dint \phi(u).
\end{align*}
The distribution of $X$ is completely determined by $\Pi$, and certain statistics of the process are described by the geometric properties of this deterministic convex body. For example, the number of hyperplanes of $X$ that have a nonempty intersection with a one-dimensional line segment $[-v,v]$ has a Poisson distribution with parameter given by the support function of $\Pi$ at $v$:
\begin{align}\label{e:zonoid_supp}
\EE[\#\{X \cap [-v,v]\}] = h_{\Pi}(v).    
\end{align}
The associated zonoid is also closely related to the distribution of the \emph{zero cell} of $X$. This is the cell of the induced random tessellation that contains the origin, defined by
\[Z_0 := \bigcap_{i \in \mathbb{N}} H^{-}(u_i,t_i),\]
where $H^{-}(u,t)$ is the closed half-space defined by the hyperplane $H(u,t)$ that contains the origin. It is known, for instance, that the expected volume of the zero cell is a constant times the volume of the polar dual of the associated zonoid \cite[(10.49)]{schneider2008stochastic}.
%\VC{Are these properties derived from the associated zonoid?  If yes, it would be useful to mention this point and give a citation.  Also, should this paragraph go in the previous subsubsection?} \EO{I've moved this up and edited as you suggested - does it look ok?}

%There exist many results on the distribution of various geometric properties of this random convex polytope including its volume and f-vector. 

\subsubsection{Poisson hypersurface model}
We will now define a model for a collection of random hypersurfaces that generalizes the construction of a stationary Poisson hyperplane process and show that we can define an associated operanoid that satisfies a property analogous to \eqref{e:zonoid_supp}.

First, let $\Phi = \{t_j\}_{j \in \mathbb{N}}$ be a Poisson point process with intensity $\gamma$ on $[0,\infty)$.
For each $j \in \mathbb{N}$, let $\{U^{(j)}_i\}_{i=1}^d$ be an independent collection of $d$ random matrices in $\mathbb{H}^m$ such that $\sum_{i=1}^d (U^{(j)}_i)^{\star}U^{(j)}_i = \sum_{i=1}^d U^{(j)}_i(U^{(j)}_i)^{\star} = I$ and that is independent of $\Phi$.
%\VC{Is there a natural random matrix ensemble that satisfies this condition?  In principle, given any random $U_i$ and then applying something called `operator scaling' does the job.  Do we want to mention this?} \EO{I don't know of a random matrix ensemble besides normalized i.i.d. unitary matrices. Could you detail what you mean by 'operator scaling' more?}
Define the random hypersurface
\begin{align}\label{e:poisson-hypersurface}
X_m = \left\{ x \in \RR^d : \prod_{j \in \mathbb{N}} \det\left(m t_jI - \sum_{i=1}^d x_i U^{(j)}_i\right) = 0\right\}.
\end{align}
We then have the following proposition.
\begin{proposition}
   For the random hypersurface $X_m$ as defined above, there exists an $m$-operanoid $\Pi$ such that for all $v \in \RR^d$,
   \begin{align*}
        \EE[\#(X_m \cap [-v,v])] = h_{\Pi}(v).
   \end{align*}
\end{proposition}

\begin{proof}

%First consider $v \in \RR^d$ such that $\|v\|_2 \leq \lambda^{-1}m$. 
We first note that $\left|\lambda_k\left(\sum_{i=1}^d v_i U_i\right)\right|  \geq mt_j$ for some $j \in \mathbb{N}$ and $k \in [m]$ if and only if there is an $s \in [-1, 1]$ such that
\[mt_j - \lambda_k\left(\sum_{i=1}^d sv_i U_i\right) = 0,\]
which implies $sv \in X_m$. Thus, 
\begin{align*}
    \EE[\#(X_m \cap [-v,v])] 
    &= \EE\left[\sum_{j \in \mathbb{N}}\sum_{k = 1}^m \mathbf{1}\left\{\left|\lambda_k\left(\sum_{i=1}^d v_i U^{(j)}_i\right)\right|  \geq mt_j \right\}\right] \\
    % &= \EE\left[\sum_{j \in\mathbb{N}} \sum_{k=1}^m \mathbf{1}\left\{m^{-1}\left|\lambda_k\left(\sum_{i=1}^d v_i U^{(j)}_i\right)\right|  \geq t_j \right\}\right] \\
     &= \gamma\EE\left[\int_{\RR} \sum_{k=1}^m \mathbf{1}\left\{m^{-1}\left|\lambda_k\left(\sum_{i=1}^d v_i U^{(j)}_i\right)\right|  \geq s \right\} \mathrm{d}s\right] = \frac{\gamma}{m} \EE\left[\sum_{k=1}^m \left|\lambda_k\left(\sum_{i=1}^d v_i U^{(j)}_i\right)\right|\right],
\end{align*}
which is the support function of an $m$-operanoid corresponding to the $d$-tuple of random matrices $\left(\gamma U_1, \ldots, \gamma U_d\right)$.  Here we have applied Campbell's theorem \cite[Theorem 3.1.2]{schneider2008stochastic} to obtain the third equality.
\end{proof}

%\subsubsection{Open questions}

There are many open questions related to this model. We state a few here, beginning with a fundamental question relating the $m$-operanoid defined in the above proposition back to the distribution of the random hypersurface.

\begin{question} Is there a one-to-one correspondence between $m$-operanoids and the distribution of a Poisson hypersurface as defined in \eqref{e:poisson-hypersurface}?
\end{question}

Another question is whether we can extend the connection between the associated zonoid and the zero cell of a Poisson hyperplane process to these random hypersurfaces. Indeed, we can analogously define the \emph{zero chamber} of the Poisson hypersurface model as the set
\[S_0 := \bigcap_{j \in \mathbb{N}} \left\{x \in \RR^d : m t_jI - \sum_{i=1}^d x_i U^{(j)}_i \succcurlyeq 0 \right\}.\]
Random spectrahedra defined by Gaussian random matrices have been previously studied in \cite{BreidingRandSpect2019}. The questions considered there inspire a similar study of $S_0$ summarized by the following open question.

\begin{question} How does the distribution of the volume and boundary structure of the zero chamber of a Poisson hypersurface depend on the properties of the associated $m$-operanoid?
\end{question}

\section{Discussion and Future Work}\label{sec:conclusion}

We introduced a class of convex bodies called operatopes obtained by taking Minkowski sums of affine images of operator norm balls.  The limits of operatopes in the Hausdorff metric are called operanoids.  These definitions generalize those of zonotopes and zonoids, which are structured convex bodies with far-reaching applications.  Operanoids may be viewed as providing convex-geometric summaries of tuples of random matrices, and based on this perspective, we describe NC zonoids which provide convex-geometric summaries of a collection of noncommutative random variables.  We established basic properties of all of these convex bodies, and we took some preliminary steps towards potential applications in statistics and stochastic processes, generalizing previous methods involving zonoids.  We conclude by listing some open questions pertaining to convex-geometric, combinatorial, and computational aspects of operanoids and NC zonoids:

\paragraph{\bf Explicit constructions} What are families of free zonoids with closed-form descriptions of the associated support functions?  Are there NC zonoids that are not free zonoids, but which can be expressed as limits of operanoids?  Alternatively, are there NC zonoids that cannot be expressed as limits of operanoids?  %Are there explicit examples of $m$-operanoids that are not $m$-operatopes? 
% \VC{I feel that this last question could have an easy answer, so I propose we remove it.  Even in two dimensions, any symmetric convex body that is not semialgebraic will do -- it is an operanoid as all symmetric convex bodies in two dimensions are zonoids, and it is not an operatope as operatopes must be semialgebraic.} \EO{I agree, I've commented it out.}

\paragraph{\bf Facial structure questions} Is there a facial characterization of operatopes akin to that of zonotopes as polytopes that have centrally symmetric faces \cite{coxeter1962classification, shephard1974combinatorial}?  What can we say about the facial structure of NC zonoids and free zonoids?

\paragraph{\bf Optimization questions} Operatopes are a tractable class of convex bodies from the viewpoint of optimization, as linear optimization over an operatope may be accomplished via eigenvalue computation.  Can we approximately solve optimization problems over operanoids using spectral computations by leveraging Theorem~\ref{t:approx_rate}?  More broadly, what are families of convex bodies that can be well-approximated by operanoids, as these would also be amenable to efficient methods for approximate optimization via Theorem~\ref{t:approx_rate}?

\paragraph{\bf Geometric questions} Zonotopes are dual to hyperplane arrangements \cite{shephard1974combinatorial}; what is the analogous duality for operatopes?  Are there specializations of the Brunn-Minkowski inequality for operanoids, akin to those for zonoids \cite{van2023local}?  Is there an analog of the notion of the projection body \cite{bourgain2006projection} that is adapted to operanoids?

\subsection*{Acknowledgements}

The authors would like to thank Jorge Garza-Vargas and Oscar Leong for helpful discussions. EO was supported in part by NSF Grant DMS-2402234. VC was supported in part by NSF grant DMS-2502377 and by AFOSR grant FA9550-23-1-0204.

\bibliographystyle{plain}
\bibliography{biblio}

\end{document}